\documentclass{article}
\usepackage{amsmath}
\usepackage{amssymb}
\usepackage{amsthm}
\usepackage{amscd}
\usepackage{bm}
\usepackage[all]{xy}
\usepackage{mathrsfs}
\usepackage{graphicx}
\usepackage{color}
\usepackage{enumitem}
\usepackage{mathtools}
\numberwithin{equation}{section}
\newtheorem{thm}{Theorem}[section]
\newtheorem{lem}[thm]{Lemma}
\newtheorem{prp}[thm]{Proposition}
\newtheorem{cor}[thm]{Corollary}

\theoremstyle{definition}
\newtheorem{dfn}[thm]{Definition}

\theoremstyle{remark}
\newtheorem{rem}[thm]{Remark}

\newcommand{\eqdot}{\coloneqq}

\newcommand{\N}{\mathbb{N}}

\newcommand{\R}{\mathbb{R}}

\newcommand{\nonnegaR}{\mathbb{R}_{\geq 0}}

\newcommand{\supover}[1]{\sup_{#1}}
\newcommand{\limtoinfty}[1]{\lim_{#1 \to \infty}}
\newcommand{\liminftoinfty}[1]{\liminf_{#1 \to \infty}}
\newcommand{\limsuptoinfty}[1]{\limsup_{#1 \to \infty}}

\newcommand{\abs}[1]{\left| #1 \right|}
\newcommand{\norm}[1]{\left\lVert #1 \right\rVert}
\mathtoolsset{showonlyrefs=true}
\title{Discrete approximation of semi-Dirichlet forms with non-symmetric perturbations of resistance forms}
\author{Hitoshi Ito}
\date{}
\allowdisplaybreaks[4]
\begin{document}
\maketitle
\begin{abstract}
    In this paper, we study discrete approximations of semi-Dirichlet forms obtained by adding non-symmetric drift terms, expressed in terms of mutual energy measures, to resistance forms whose associated resistance metric spaces are compact. The main novelty is that the drift terms need not to be absolutely continuous with respect to the underlying measure. Under suitable smallness assumptions, we prove generalized Mosco convergence of the approximating forms and derive weak convergence of the associated Feller processes. We also demonstrate the applicability of our main theorem to post-critically finite self-similar fractals and describe the resulting perturbation of the approximating jump chains in the case of the Sierpi\'nski gasket.
\end{abstract}
\section{Introduction}
Dirichlet forms, a type of quadratic forms defined on $L^{2}$ spaces, are fundamental tools for analyzing Markov processes. For instance, the symmetric Dirichlet form $(\mathcal{E},\mathcal{F})$ defined by 
\begin{align}\label{normaldirich}
    \mathcal{E}(f,g) \coloneqq \frac{1}{2}\int_{\R^{d}}(\nabla f, \nabla g)_{\R^{d}} \,dx, 
    \quad f,g \in \mathcal{F} \coloneqq H^{1}(\R^{d})
\end{align}
is associated with the Brownian motion on $\R^{d}$ whose generator is the half Laplacian on $\R^{d}$, where $( \cdot, \cdot)_{\R^{d}}$ denotes the canonical inner product on $\R^{d}$ and $H^{1}(\R^{d})$ denotes the first-order $L^2$-Sobolev space on $\R^{d}$. Dirichlet forms have also been used for the analysis on fractals. Around 1990, Kigami \cite{kigami1989harmonic} and Fukushima--Shima \cite{fukushima1992spectral} used 
symmetric Dirichlet forms for the construction of the Brownian motion on the Sierpi\'nski gasket, which, in the late 1980s, had been constructed by Goldstein \cite{goldstein1987random}, Kusuoka \cite{MR933827}, and 
Barlow--Perkins \cite{barlow1988brownian} in probabilistic approaches such as the scaling limits of random walks on finite sets approximating the Sierpi\'nski 
gasket. The Dirichlet form constructed on the Sierpi\'nski gasket is considered as the counterpart of \eqref{normaldirich}.
Also, Kigami \cite{kigami1993harmonic} 
extended the above argument to post-critically finite (p.c.f. for short) 
self-similar fractals.

To date, including the aforementioned pioneering works, studies of diffusions on fractals have
mainly focused on symmetric Dirichlet forms and relatively few studies have treated 
non-symmetric Dirichlet forms on fractals. Since the general theory of such concepts 
is now well-developed, it would be worth investigating non-symmetric 
Dirichlet forms (which we refer to as semi-Dirichlet forms in this paper) on fractals and 
associated diffusion processes. We focus in particular on their discrete approximation. 

More specifically, we treat non-symmetric forms as described below. 
On $\R^{d}$, the following quadratic form obtained by adding a drift term to \eqref{normaldirich}
\begin{align} \label{eucnonsym}
    (f,g) \mapsto \frac{1}{2}\int_{\R^{d}}(\nabla f, \nabla g)_{\R^{d}} \,dx
    + \int_{\R^{d}}(b, \nabla f)_{\R^{d}}g \,dx
\end{align}
naturally arises, where $b$ is an $\R^{d}$-valued function.
We would like to consider the counterpart of \eqref{eucnonsym} on fractals, where
we usually cannot define the first order derivative.
We first consider the case of the Sierpi\'nski gasket SG. Let $(\mathcal{E},\mathcal{F})$ 
be the canonical Dirichlet form on the Sierpi\'nski gasket. Despite the fact that 
the Sierpi\'nski gasket does not have a usual differential structure, \cite[Theorem 5.4]{hino2010energy} shows that if we take a 
suitable function $h \in \mathcal{F}$, we can define the 
\textit{first-order derivative} $\frac{df}{dh}$ of $f$ with respect to $h$ for all 
$f \in \mathcal{F}$ and we have the following identity: 
\begin{align}
    \mathcal{E}(f,f)
= \frac{1}{2}
\int_{\mathrm{SG}} \left(\frac{df}{dh} \right)^{2} \,d\nu_{h},
\end{align}
where $\nu_{h}$ denotes the energy measure of $h$ (see Definition~\ref{defitionenergymeasure}). Thus, 
when we consider $\mathcal{E}(f,g)$ as the first term of \eqref{eucnonsym}, it is 
natural to think that the counterpart of the second drift term of \eqref{eucnonsym} 
would be 
\begin{align} \label{fracdifterm}
    \int_{\mathrm{SG}} b\frac{df}{dh}g \,d\nu_{h},
\end{align}
where $b$ is an $\R$-valued function.
Taking into consideration that the equation
\begin{align} 
    \frac{df}{dh}=\frac{d\nu_{f,h}}{d\nu_{h}}
\end{align}
holds, where $\nu_{f,h}$ denotes the mutual energy measure of $f$ and $h$, 
the term \eqref{fracdifterm} equals the term
\begin{align} \label{fracmutenergterm}
    \int_{\mathrm{SG}} bg \,d\nu_{f,h}
\end{align}
in which no \textit{first-order derivatives} appear. Note that 
studies investigating diffusions on non-smooth spaces associated with non-symmetric Dirichlet forms of type \eqref{eucnonsym} already exist in the context of 
metric measure spaces. For example, in \cite{suzuki2018convergence}, non-symmetric drift terms of the form
\begin{align} \label{absolut}
    \int_{X} b(f)g \,dm
\end{align}
are considered in a metric measure space $(X,d, m)$, where $b$ is called a derivation in \cite{suzuki2018convergence}. What is important is that the drift term
\eqref{absolut} is absolutely continuous with respect to the underlying measure $m$. 
On the other hand, the energy measure in \eqref{fracmutenergterm} is singular with respect to the 
Hausdorff measure on the Sierpi\'nski gasket (see \cite{kusuoka1989dirichlet}). The singularity of the energy measure with respect to the underlying measure is well-known not only for the Sierpi\'nski gasket but also for a wide class of fractals (see e.g.\ \cite{kusuokanested}, \cite{bstsingular}, \cite{hinosingular}, and \cite{kajinosingular}).
In this sense, the setting is different from the studies in \cite{suzuki2018convergence}. Based on the above considerations 
in the case of the Sierpi\'nski gasket, we adopt non-symmetric terms
\eqref{fracmutenergterm} of this type and consider semi-Dirichlet forms obtained by adding them to a symmetric term. Our ultimate goal is to reveal concrete behavior of diffusion processes associated with such semi-Dirichlet forms. Since they do not have explicit representations such as stochastic differential equations, we investigate their behaviors through discrete approximations in this article.

Now, we briefly explain the setting of the problem. Let $X$ be a set (not necessarily a fractal), $(\mathcal{E},\mathcal{F})$ a resistance form on $X$.
Assume that $X$ is compact with respect to the resistance metric associated with $(\mathcal{E},\mathcal{F})$. 
Let $b_{1},\dots,b_{N} \in C(X)$ and $h_{1},\dots,h_{N} \in \mathcal{F}$. 
We add the following $N$ drift terms
\begin{equation}
\mathcal{Q}_{i}(f,g) \eqdot \frac{1}{2}\int_{X} b_{i}g \,d\nu_{f,h_{i}}
, \quad f,g \in \mathcal{F}, \quad i=1,2, \dots ,N
\end{equation}
to $(\mathcal{E}, \mathcal{F})$ and obtain a non-symmetric form $(\mathcal{A}, \mathcal{F})$ by
\begin{align}
\mathcal{A} \eqdot \mathcal{E}+\sum_{i=1}^{N}\mathcal{Q}_{i}.
\end{align}
Let
$\{ V_{n} \}_{n \in \N}$ be a non-decreasing sequence of nonempty finite 
subsets approximating $X$. For each $n \in \N$, we suitably define 
$\mathcal{A}^{n} \colon l(V_{n}) \times l(V_{n}) \rightarrow \R$ as 
a discrete approximation of $\mathcal{A}$ on $V_{n}$, 
where $l(V_{n})$ denotes the set of all real functions on $V_{n}$.
Let $\mu$ (resp.\ $\mu_{n}$) be a probability measure on $X$ (resp.\ $V_{n}$) and assume that
$\{ \mu_{n} \}_{n \in \N}$ converges weakly to $\mu$. 
Then, under Assumption A or B in Section 3, it is shown that $(\mathcal{A}, \mathcal{F})$ (resp.\ $(\mathcal{A}^{n},l(V_{n}))$) 
is a regular semi-Dirichlet form on $L^2(X,\mu)$ (resp.\ $L^2(V_{n},\mu_{n})$). Let $\mathbb{M}=((Z_{t})_{t \ge0},(P_{x})_{x \in X})$ (resp.\ $\mathbb{M}^{n}=((Y^{n}_{t})_{t \ge0},(P^{n}_{x})_{x \in V_{n}})$) be a Hunt process on $X$ (resp.\ $V_{n}$) associated with the regular semi-Dirichlet form $(\mathcal{A},\mathcal{F})$ (resp.\ $(\mathcal{A}^{n},l(V_{n}))$) (see Proposition~\ref{semidirichletassociatehunt} for details). Let $\nu$ $($resp.\ $\nu_{n}$$)$ be a probability measure on $X$ $($resp.\ $V_{n}$$)$. We define $P_{\nu}$ (resp.\ $P^{n}_{\nu_{n}}$) by $P_{\nu}(\cdot) \coloneqq \int_{X}P_{x}(\cdot)\,\nu(dx)$ (resp.\ $P^{n}_{\nu_{n}}(\cdot) \coloneqq \int_{V_{n}}P^{n}_{x}(\cdot)\,\nu_{n}(dx)$). Let $D_{X}[0,\infty)$ denote the space of all $X$-valued c\`adl\`ag functions on $[0,\infty)$.
Our main results are stated as follows.
\begin{thm}\label{mainfeller}
    Assume that Assumption A or Assumption B holds. Then, $\mathbb{M}$ $($resp.\ $\mathbb{M}^{n}$$)$ is conservative Feller.
\end{thm}
\begin{thm}\label{mainfirst}
Assume that Assumption A or Assumption B holds and $\{ \nu_{n} \}_{n \in \N}$ converges weakly to $\nu$. Let $i_{V_{n}} \colon V_{n}\rightarrow X$ be 
the canonical inclusion map and define $Z^{n}_{t} \coloneqq i_{V_{n}} \circ Y^{n}_{t}$. Then, $\{P^{n}_{\nu_{n}}\left((Z^{n}_{t})_{t \ge 0} \in \cdot \right) \}_{n \in \N}$ converges weakly to $P_{\nu} \left((Z_{t})_{t \ge 0} \in \cdot \right)$ as distributions on $D_{X}[0,\infty)$ equipped with the Skorokhod J1 topology.
\end{thm}
\begin{rem}
Since $(\mathcal{A}^{n},l(V_{n}))$ is a semi-Dirichlet form only for large enough 
$n \in \N$ (see Proposition~\ref{semidir1}), we consider $(Z^{n}_{t})_{t \ge 0}$ only for large enough 
$n \in \N$ in Theorem~\ref{mainfirst}.
\end{rem}

Thus, the contribution of this paper is that we introduce a class of semi-Dirichlet forms whose drift terms are described by mutual energy measures, which need not to be absolutely continuous with respect to the underlying measure and discuss the convergence of their discrete approximations. The main tool for the proof is a generalized version of Mosco convergence. 
Mosco \cite{MR1283033} discussed the equivalence of convergence of symmetric quadratic forms and 
strong convergence of associated resolvents. Hino \cite{hino1998convergence} generalized this in the case of non-symmetric quadratic forms. T\"{o}lle \cite{Jtolle} 
treated the convergence of non-symmetric quadratic forms on varying spaces in 
Kuwae--Shioya's framework which was discussed in \cite{kuwae2003convergence}. 
We use T\"{o}lle's version.

The remainder of the article is organized as follows. In Section 2, we review Kuwae--Shioya's 
framework of varying spaces which was generalized by T\"olle and the basics of semi-Dirichlet forms and resistance forms. 
In Section 3, we describe the setting of the problem in detail.
In Section 4, we prove Theorem~\ref{mainfirst}. In Section 5, we discuss p.c.f.\ fractals 
as typical examples to which Theorem~\ref{mainfirst} can be applied.
\section{Preliminaries}
\subsection{Kuwae--Shioya's framework}
Throughout this paper, all normed spaces are assumed to be over the real field. For a normed space $\mathscr{X}$, 
$\norm{\cdot}_{\mathscr{X}}$ denotes its norm. For an inner product space $\mathscr{H}$, 
$( \cdot, \cdot)_{\mathscr{H}}$ denotes its inner product and its norm is defined by 
$\norm{\cdot}_{\mathscr{H}} \eqdot ( \cdot, \cdot)_{\mathscr{H}}^{1/2}$. 
For normed spaces $\mathscr{X}^{(i)} \,(i=1,2)$, $\mathscr{L}(\mathscr{X}^{(1)}, \mathscr{X}^{(2)})$
denotes the set of all bounded linear operators from $\mathscr{X}^{(1)}$ to $\mathscr{X}^{(2)}$ and 
$\norm{B}_{\mathrm{op}}$ denotes the operator norm of 
$B \in \mathscr{L}(\mathscr{X}^{(1)}, \mathscr{X}^{(2)})$. Also, we define
$\mathscr{L}(\mathscr{X})
\eqdot\mathscr{L}(\mathscr{X}, \mathscr{X})$.

We follow \cite{kuwae2003convergence, Jtolle} to provide Kuwae--Shioya's framework. In this section, 
all the proofs are omitted. For details, see \cite{kuwae2003convergence, Jtolle}. 
In order to distinguish the convergence in Kuwae--Shioya's framework defined below from the usual one, starting in Section 3, we prefix ``K-S'' to ``convergence'' 
whenever we mention the convergence in Kuwae--Shioya's framework.
\begin{dfn}[{\cite[p.\ 611]{kuwae2003convergence}}] \label{ksspaceconve}
Let $\mathscr{X}_{n} (n \in \N)$ and $\mathscr{X}_{\infty}$ be Banach spaces, 
$\mathscr{C}_{\infty}$ a dense subspace of $\mathscr{X}_{\infty}$ and 
$\Phi_{n} \colon \mathscr{C}_{\infty} \rightarrow \mathscr{X}_{n}$ a linear map. We say that 
$\{ \mathscr{X}_{n} \}_{n \in \N}$ converges to $\mathscr{X}_{\infty}$ if the following holds:
\begin{equation} \label{banachconv}        
    \limtoinfty{n} \norm{\Phi_{n} \varphi}_{\mathscr{X}_{n}} = \norm{\varphi}_{\mathscr{X}_{\infty}},
    \quad \varphi \in \mathscr{C}_{\infty}.
    \end{equation}
\end{dfn}
In this section, let $\mathscr{X}_{n} (n \in \N)$ and $\mathscr{X}_{\infty}$ be Banach spaces, 
$\mathscr{C}_{\infty}$ a dense subspace of $\mathscr{X}_{\infty}$ and 
$\Phi_{n} \colon \mathscr{C}_{\infty} \rightarrow \mathscr{X}_{n}$ a linear map, 
and we assume that the sequence of Banach spaces
$\{ \mathscr{X}_{n} \}_{n \in \N}$ converges to a Banach space $\mathscr{X}_{\infty}$ 
with respect to $\mathscr{C}_{\infty}$ and $\{ \Phi_{n} \}_{n \in \N}$. We fix them throughout the discussion in this section. When we assume that $\mathscr{X}_{n} (n \in \N)$ and $\mathscr{X}_{\infty}$ are Hilbert spaces in addition, we use $\ \mathscr{H}_{n}(n \in \N)$ and $\mathscr{H}_{\infty}$ instead of $\mathscr{X}_{n} (n \in \N)$ and $\mathscr{X}_{\infty}$.
\begin{dfn}[{\cite[Definition 2.4]{kuwae2003convergence}}]
We say that $\{ u_{n} \}_{n \in \N} \,(u_{n} \in \mathscr{X}_{n})$ strongly converges to 
$u \in \mathscr{X}_{\infty}$ if there exists 
$\{\varphi_{m} \}_{m \in \N} \subset \mathscr{C}_{\infty}$ which satisfies 
\begin{equation}
    \limtoinfty{m} \norm{\varphi_{m}-u}_{\mathscr{X}_{\infty}}=0 
    \text{ and }
    \limtoinfty{m} \limsuptoinfty{n}
    \norm{\Phi_{n} \varphi_{m}-u_{n}}_{\mathscr{X}_{n}}=0.
\end{equation}
\end{dfn}
\begin{rem}\label{rem3}
For any $\varphi \in \mathscr{C}_{\infty}$, $\Phi_{n} \varphi \in \mathscr{X}_{n}$ 
strongly converges to $\varphi \in \mathscr{X}_{\infty}$
by letting $\varphi_{m} \eqdot \varphi \in \mathscr{C}_{\infty}$.
\end{rem}
\begin{prp}[{\cite[Lemma 2.1]{kuwae2003convergence}, 
\cite[Lemma 2.12]{Jtolle}}]\label{kyo1}
    The following hold.
\begin{enumerate}[label=\normalfont(\roman*)]
    \item \label{enum:2kyo}
    If $\{ u_{n} \}_{n \in \N} \,(u_{n} \in \mathscr{X}_{n})$ 
    strongly converges to $u \in \mathscr{X}_{\infty}$, then 
    $\limtoinfty{n} \norm{u_{n}}_{\mathscr{X}_{n}}$ $=\norm{u}_{\mathscr{X}_{\infty}}$.
    \item \label{enum:4kyo}
    Assume that $\{ u_{n} \}_{n \in \N} \,(u_{n} \in \mathscr{X}_{n})$
    strongly converges to $u \in \mathscr{X}_{\infty}$. Then, 
    for $\{ v_{n} \}_{n \in \N} \,(v_{n} \in \mathscr{X}_{n})$, $\{ v_{n} \}_{n \in \N}$
    strongly converges to $u \in \mathscr{X}_{\infty}$ if and only if 
    $\limtoinfty{n} \norm{u_{n}-v_{n}}_{\mathscr{X}_{n}}=0$.
    \item \label{enum:7kyo}
    If $\{ u^{(i)}_{n} \}_{n \in \N} \,(u^{(i)}_{n} \in \mathscr{H}_{n})$ strongly converges to
    $u^{(i)} \in \mathscr{H}_{\infty}$ for $i=1,2$, then 
    $\limtoinfty{n}(u^{(1)}_{n},u^{(2)}_{n})_{\mathscr{H}_n}=
    (u^{(1)}, u^{(2)})_{\mathscr{H}_{\infty}}$.
    \item \label{enum:8kyo}
    $\{ u_{n} \}_{n \in \N} \,(u_{n} \in \mathscr{H}_{n})$ strongly converges to 
    $u \in \mathscr{H}_{\infty}$ if and only if 
    \begin{align}
    \limtoinfty{n} 
    \norm{u_{n}}_{\mathscr{H}_{n}}=\norm{u}_{\mathscr{H}_{\infty}} \text{and}
    \limtoinfty{n}(u_{n}, \Phi_{n}v)_{\mathscr{H}_{n}}=(u,v)_{\mathscr{H}_{\infty}}, 
    \quad v \in \mathscr{C}_{\infty}.
    \end{align}
\end{enumerate}
\end{prp}
By Remark~\ref{rem3} and Proposition~\ref{kyo1}~\ref{enum:4kyo}, the following holds:
\begin{cor} \label{kyocor}
Assume that $\mathscr{C}_{\infty}=\mathscr{X}_{\infty}$. Then, 
$\{ u_{n} \}_{n \in \N} \,(u_{n} \in \mathscr{X}_{n})$ strongly converges to $u \in \mathscr{X}_{\infty}$ 
if and only if $\limtoinfty{n} \norm{\Phi_{n}u-u_{n}}_{\mathscr{X}_{n}}=0$.
\end{cor}
\begin{dfn}[cf.\ {\cite[Definition 2.5]{kuwae2003convergence}}]\label{weaklyconvergekuwaeshioya}
We say that $\{ u_{n} \}_{n \in \N} \,(u_{n} \in \mathscr{H}_{n})$ weakly converges to $u \in \mathscr{H}_{\infty}$
if the following condition holds: $\supover{n \in \N} \norm{u_n}_{\mathscr{H}_{n}} < \infty$, and $\limtoinfty{n} (u_{n}, v_{n})_{\mathscr{H}_{n}}=(u, v)_{ \mathscr{H}_{\infty}}$ holds for any $v \in \mathscr{H}_{\infty}$ and $\{ v_{n} \}_{n \in \N} \,(v_{n} \in \mathscr{H}_{n})$ strongly converging to $v$. 
\end{dfn}
\begin{rem}
In Definition~\ref{weaklyconvergekuwaeshioya}, we add the condition $\supover{n \in \N} \norm{u_n}_{\mathscr{H}_{n}} < \infty$ 
unlike the original definition \cite[Definition 2.5]{kuwae2003convergence}.
\end{rem}
\begin{rem} \label{rem5}
The strong convergence implies the weak convergence by Proposition~$\ref{kyo1}$~\ref{enum:2kyo} 
and \ref{enum:7kyo}.
\end{rem}
\begin{prp}[{\cite[Lemma 2.13]{Jtolle}}]\label{wk1}
The following hold.
\begin{enumerate}[label=\normalfont(\roman*)]
    \item \label{enum:1wk}
    $\{ u_{n} \}_{n \in \N} \,(u_{n} \in \mathscr{H}_{n})$ weakly converges to 
    $u \in \mathscr{H}_{\infty}$ if and only if $\supover{n} \norm{u_{n}}_{\mathscr{H}_{n}}$ $< \infty$ and $\limtoinfty{n}(u_{n}, \Phi_{n}v)_{\mathscr{H}_{n}}=(u,v)_{\mathscr{H}_{\infty}}$ for any $v \in \mathscr{C}_{\infty}$.
    \item \label{enum:4wk}
    If $\{ u_{n} \}_{n \in \N} \,(u_{n} \in \mathscr{H}_{n})$ weakly converges to 
    $u \in \mathscr{H}_{\infty}$, then for any subsequence $\{ u_{n_{k}} \}_{k \in \N}$ of 
    $\{ u_{n} \}_{n \in \N}$, $\{ u_{n_{k}} \}_{k \in \N}$ weakly converges to $u$.
    \item \label{enum:5wk}
    If $\{ u_{n} \}_{n \in \N} \,(u_{n} \in \mathscr{H}_{n})$ weakly converges to both 
    $u^{(1)}$ and $u^{(2)} \in \mathscr{H}_{\infty}$, then $u^{(1)}=u^{(2)}$ holds.
\end{enumerate}
\end{prp}
\begin{dfn}[{\cite[Definition 2.6]{kuwae2003convergence}}]
We say that $\{ B_{n} \}_{n \in \N} \,( B_{n} \in \mathscr{L}(\mathscr{X}_{n}))$ strongly $($resp.\ weakly$)$ 
converges to $B_{\infty} \in \mathscr{L}(\mathscr{X}_{\infty})$ if the following holds: 
The strong $($resp.\ weak$)$ convergence of 
$\{ u_{n} \}_{n \in \N} \,(u_{n} \in \mathscr{X}_{n})$ to $u \in \mathscr{X}_{\infty}$ implies 
the strong $($resp.\ weak$)$ convergence of 
$\{ B_{n}u_{n} \}_{n \in \N} \,(B_{n}u_{n} \in \mathscr{X}_{n})$ to 
$B_{\infty}u \in \mathscr{X}_{\infty}$.
\end{dfn}
\begin{lem}[{\cite[Lemma 2.20]{Jtolle}}]\label{tech3}
Assume $\mathscr{C}_{\infty}=\mathscr{X}_{\infty}$ holds and $\mathscr{C}^{\prime}_{\infty}$ is 
a dense subspace of $\mathscr{X}_{\infty}$. Let $B_{n} \in \mathscr{L}(\mathscr{X}_{n})$, 
$B_{\infty} \in \mathscr{L}(\mathscr{X}_{\infty})$ and assume that 
$\supover{n}\norm{B_{n}}_{\mathrm{op}} < \infty$ and 
$\{ B_{n} \Phi_{n} \psi \}_{n \in \N}
\,( B_{n} \Phi_{n} \psi \in \mathscr{X}_{n})$ strongly converges to 
$ B_{\infty}\psi \in \mathscr{X}_{\infty}$ for any $\psi \in \mathscr{C}^{\prime}_{\infty}$. 
Then, $\{ B_{n} \}_{n \in \N}$ strongly converges to $B_{\infty}$.
\end{lem}
\begin{prp}[{\cite[Theorem 2.21]{Jtolle}}]\label{daizi}
Assume that $\mathscr{C}_{\infty}=\mathscr{X}_{\infty}$ and 
$\supover{n}\norm{\Phi_{n}}_{\mathrm{op}}$ $< \infty$ $($$\Phi_{n} \in \mathscr{L}(\mathscr{X}_{\infty},\mathscr{X}_{n})$ is implied$)$. 
Let $\{T_{t} \}_{t \geq0}$ $($resp.\ $\{T^{(n)}_{t} \}_{t \geq0}$$)$ be a strongly continuous 
contraction semigroup on $\mathscr{X}_{\infty}$ $($resp.\ $\mathscr{X}_{n}$$)$ and 
$\{G_{\alpha} \}_{\alpha >0}$ $($resp.\ $\{G^{(n)}_{\alpha} \}_{\alpha >0}$$)$ be its 
resolvent on $\mathscr{X}_{\infty}$ $($resp.\ $\mathscr{X}_{n}$$)$. 
Then, $\{T^{(n)}_{t} \}_{n \in \N} \,(T^{(n)}_{t}\in \mathscr{L}(\mathscr{X}_{n}))$ 
strongly converges to $T_{t}\in \mathscr{L}(\mathscr{X}_{\infty})$ for any $t \geq 0$ 
if $\{G^{(n)}_{\alpha} \}_{n \in \N}\,(G^{(n)}_{\alpha}\in \mathscr{L}(\mathscr{X}_{n}))$ 
strongly converges to $G_{\alpha}\in \mathscr{L}(\mathscr{X}_{\infty})$ for some $\alpha >0$.
\end{prp}
\subsection{Semi-Dirichlet forms}
We mainly follow \cite{oshima2013semi} for the definition and basic properties of 
semi-Dirichlet forms. For any $a,b \in \R$, we set $a \wedge b =\mathrm{min}\{a,b\}$
 and $a \vee b =\mathrm{max}\{a,b\}$. Let $\mathscr{H}$ be a Hilbert space, $\mathcal{F}$ a dense subspace of $\mathscr{H}$, and 
$\mathcal{A} \colon \mathcal{F} \times \mathcal{F} \rightarrow \R$ a bilinear form. 
For $\alpha \geq 0$, we define a bilinear form
$\mathcal{A}_{\alpha} \colon \mathcal{F} \times \mathcal{F} \rightarrow \R$ by 
\begin{equation}
\mathcal{A}_{\alpha}(u,v)
\coloneqq
\mathcal{A}(u,v)
+\alpha (u,v)_{\mathscr{H}}, \quad u,v \in \mathcal{F}.
\end{equation}
Also, we define $\mathcal{A}(u) \coloneqq \mathcal{A}(u,u)$ for $u \in \mathcal{F}$.
\begin{dfn}[{\cite[p.\ 1]{oshima2013semi}}] 
$(\mathcal{A},\mathcal{F})$ is called a closed form on $\mathscr{H}$ if the following hold: 
\begin{enumerate}
    \item[$\mathrm{(SD1)}$]
    There exists $\lambda \geq 0$ such that $\mathcal{A}_{\lambda}(u) \geq 0$ 
    for any $u \in  \mathcal{F}$.
    \item[$\mathrm{(SD2)}$]
    $\mathcal{A}_{\lambda+1}(\cdot)^{1/2}$ is a complete norm on $\mathcal{F}$.
    \item[$\mathrm{(SD3)}$]
    There exists $K \geq 1$ such that
    \begin{equation}
    \abs{\mathcal{A}(u,v)}
    \leq
    K\mathcal{A}_{\lambda}(u)^{1/2}\mathcal{A}_{\lambda}(v)^{1/2}, 
    \quad u,v \in  \mathcal{F}.
    \end{equation}
\end{enumerate}
When $\mathscr{H} = L^2(X,\mu)$ with a measure space $(X,\mu)$, 
a closed form $(\mathcal{A},\mathcal{F})$ on $L^2(X,\mu)$ is called a semi-Dirichlet form 
on $L^2(X,\mu)$ if the following holds in addition:
\begin{enumerate}[resume]
\item[$\mathrm{(SD4)}$]
For any $f \in \mathcal{F}$ and $a \geq 0$, we have $f \wedge a \in \mathcal{F}$ 
and $\mathcal{A}(f \wedge a,f-f \wedge a) \geq 0$.
\end{enumerate}
Condition $\mathrm{(SD4)}$ is called the Markov property. 

Also, a semi-Dirichlet form $(\mathcal{A},\mathcal{F})$ on $L^2(X,\mu)$ is called 
a Dirichlet form on $L^2(X,\mu)$ if $\mathcal{A}$ is symmetric and 
we can take $0$ as the constant $\lambda$ in $\mathrm{(SD1)}$.
\end{dfn}
\begin{prp}[{\cite[Theorem 1.1.2]{oshima2013semi}}]
Let $(\mathcal{A},\mathcal{F})$ be a closed form on $\mathscr{H}$ with the constant $\lambda$ in $\mathrm{(SD1)}$. Then, there exists a unique strongly continuous 
semigroup $\{T_{t} \}_{t \geq0}$ on $\mathscr{H}$ such that 
$\norm{T_{t}}_{\mathrm{op}} \leq e^{\lambda t}$ for $t\geq0$, and its resolvent 
defined by 
\begin{equation}
    G_{\alpha}u \coloneqq
 \int_{0}^{\infty}e^{-\alpha t}T_{t}u \,dt \in \mathscr{H}, \quad u \in \mathscr{H}, 
 \quad \alpha > \lambda
\end{equation}
satisfies 
\begin{equation}\label{characterresolvent}
\mathcal{A}_{\alpha}(G_{\alpha}u,v)=(u,v)_{\mathscr{H}}, 
\quad u\in \mathscr{H},v \in \mathcal{F}, \alpha > \lambda.
\end{equation}
\end{prp}
\begin{rem}
Note that \eqref{characterresolvent} characterizes the above resolvent and also 
that $\norm{G_{\alpha}}_{\mathrm{op}} \leq (\alpha -\lambda)^{-1}$ holds.
\end{rem}
\begin{rem}
When $\mathscr{H} = L^2(X,\mu)$, the Markov property $\mathrm{(SD4)}$ is equivalent to 
the following $($see \textup{\cite[Theorem 1.1.5]{oshima2013semi}}$)$: If $f \in L^2(X, \mu)$ satisfies $0 \leq f \leq 1 \,\mu \text{-a.e.}$, then 
$0 \leq T_{t}f \leq 1 \,\mu \text{-a.e.}$
\end{rem}
\begin{prp}[{\cite[Theorem 2.41]{Jtolle}}]\label{rslvconv}
Assume that $\mathscr{H}_{\infty}$ is separable. Let $(\mathcal{A},\mathcal{F})$ 
$($resp.\ $(\mathcal{A}^{n},\mathcal{F}^{n})$$)$
be a closed form on $\mathscr{H}_{\infty}$ $($resp.\ $\mathscr{H}_{n}$$)$ with the constant $\lambda$ in $\mathrm{(SD1)}$ and $\{G_{\alpha} \}_{\alpha >\lambda}$
$($resp.\ $\{G^{(n)}_{\alpha} \}_{\alpha >\lambda}$$)$ its resolvent. 
Moreover, assume that the following two conditions hold:
\begin{enumerate}
\item[$\mathrm{(F1)}$]
Let $n_{k} \uparrow \infty, u_{k} \in \mathcal{F}^{n_{k}}, u \in \mathscr{H}_{\infty}$ and 
assume that $\{ u_{k} \}_{k \in \N}$ weakly converges to $u$ and 
$\supover{k} \mathcal{A}^{n_{k}}_{\lambda+1}(u_{k}) < \infty$. Then, $u \in \mathcal{F}$.
\item[$\mathrm{(F2)}$]
Let $n_{k} \uparrow \infty, u_{k} \in \mathcal{F}^{n_{k}}, u, w \in \mathscr{F}$ and 
assume that $\{ u_{k} \}_{k \in \N}$ weakly converges to $u$ and 
$\supover{k} \mathcal{A}^{n_{k}}_{\lambda+1}(u_{k}) < \infty$. Then there exists 
$\{ w_{k} \}_{k \in \N} \,(w_{k} \in \mathcal{F}^{n_{k}}
\subset \mathscr{H}_{n_{k}})$ such that $\{ w_{k} \}_{k \in \N}$ strongly converges to $w$ and 
\begin{align}
\liminftoinfty{k} \mathcal{A}^{n_{k}}(w_{k},u_{k}) \leq \mathcal{A}(w,u).
\end{align}
\end{enumerate}
Then, for any $\alpha >\lambda$, $\{G^{(n)}_{\alpha} \}_{n \in \N}
\,(G^{(n)}_{\alpha}\in \mathscr{L}(\mathscr{H}_{n}))$ strongly converges to 
$G_{\alpha}\in \mathscr{L}(\mathscr{H}_{\infty})$.
\end{prp}
Let $(X,d)$ be a locally compact separable metric space, $\mu$ a Radon measure on $X$ with full support, $(\mathcal{A},\mathcal{F})$ a semi-Dirichlet form on $L^2(X,\mu)$ with the constant $\lambda$ in $\mathrm{(SD1)}$. Let $B_{b}(X)$ denote the space of all real bounded Borel measurable functions on $(X,d)$ and $C_{0}(X)$ the space of all real continuous functions with compact support on $(X,d)$ with the uniform norm $\norm{\cdot}_{\infty}$. We take an arbitrary $\alpha > \lambda$.
\begin{dfn}[cf.\ {\cite[Definition 2]{oshima2013semi}}]\label{regularsemidef}
We say that $(\mathcal{A},\mathcal{F})$ is regular if $\mathcal{F} \cap C_{0}(X)$ is dense both in $(C_{0}(X), \norm{\cdot}_{\infty})$ and 
$(\mathcal{F}, \mathcal{A}_{\alpha}(\cdot)^{1/2})$.
\end{dfn}
\begin{rem}\label{regularremark}
If $\mathcal{F}$ is a subset of $C_{0}(X)$ and dense in 
$(C_{0}(X), \norm{\cdot}_{\infty})$, then $(\mathcal{A},\mathcal{F})$ is regular.
\end{rem}
\begin{dfn}[{\cite[Definition 3 and 4]{oshima2013semi}}]\label{capadef}
 Let $U$ be an open set in $(X,d)$. We define $\mathscr{L}_{U}$ by
\begin{align}
\mathscr{L}_{U} \coloneqq 
\{ u \in  \mathcal{F}\mid u \ge1 \, \mu \text{-a.e. on } U\}.
\end{align}
If $\mathscr{L}_{U} \neq \emptyset$, then there exist 
unique $e^{\alpha}_{U} \in \mathscr{L}_{U}$ and $\hat{e}^{\alpha}_{U} \in \mathscr{L}_{U}$ which satisfy 
\begin{align}
    \mathcal{A}_{\alpha}(e^{\alpha}_{U}, e^{\alpha}_{U}) \le\mathcal{A}_{\alpha}(e^{\alpha}_{U}, u) \text{ and } \mathcal{A}_{\alpha}(\hat{e}^{\alpha}_{U}, \hat{e}^{\alpha}_{U}) \le\mathcal{A}_{\alpha}(u, \hat{e}^{\alpha}_{U}), \quad u \in \mathscr{L}_{U}.
\end{align}
We call $e^{\alpha}_{U}$ (resp.\ $\hat{e}^{\alpha}_{U}$) the $\alpha$-equilibrium potential (resp.\ $\alpha$-coequilibrium potential) of $U$. We define the $\alpha$-capacity 
$\mathop{\mathrm{Cap}}\nolimits^{(\alpha)}(U)$ of $U$ by 
\begin{equation}
    \mathop{\mathrm{Cap}}\nolimits^{(\alpha)}(U) \coloneqq
    \begin{cases}
    \mathcal{A}_{\alpha}(e^{\alpha}_{U}, \hat{e}^{\alpha}_{U})
    & \text{if $\mathscr{L}_{U} \neq \emptyset$,}\\
    \infty & \text{if $\mathscr{L}_{U} = \emptyset$.}
    \end{cases}
\end{equation}
In addition, for any subset $B \subset X$, we define the $\alpha$-capacity $\mathop{\mathrm{Cap}}\nolimits^{(\alpha)}(B)$ of $B$ by 
\begin{align}
\mathop{\mathrm{Cap}}\nolimits^{(\alpha)}(B) \coloneqq 
\inf \{\mathop{\mathrm{Cap}}\nolimits^{(\alpha)}(U)\mid B \subset U \text{ and } U \text{ is open}\}.
\end{align}
\end{dfn}
\begin{rem}\label{symcapa}
If $(\mathcal{A},\mathcal{F})$ is symmetric in addition, 
then, for any open set $U$,
\begin{align}
    \mathop{\mathrm{Cap}}\nolimits^{(\alpha)}(U) = 
    \min \{\mathcal{A}_{\alpha}(u, u)  \mid u \in  \mathscr{L}_{U} \} \text{ holds in the case of }
    \mathscr{L}_{U} \neq \emptyset.
\end{align}
\end{rem}
\begin{dfn}[{\cite[p.\ 51]{oshima2013semi}}]\label{qcqedef}
We call a real function $f$ on $X$ quasi-continuous if, 
for any $\varepsilon >0$, there exists an open set $U$ in $X$ such that 
$\mathop{\mathrm{Cap}}\nolimits^{(\alpha)}(U) < \varepsilon$ and $f\lvert_{X \setminus U}$ is continuous. We say that a proposition $P(x)$ about $x \in X$ holds quasi-everywhere (q.e.\ for short) if there exists a subset $N$ of $X$ such that $\mathop{\mathrm{Cap}}\nolimits^{(\alpha)}(N)=0$ and $P(x)$ 
holds for any $x \in X \setminus N$.
\end{dfn}
\begin{dfn}[{\cite[p.\ 77]{oshima2013semi}}]\label{transitionresolvmarkovdef}
Let $\mathbb{M}=((Z_{t})_{t \ge0},(P_{x})_{x \in X})$ be a Hunt process on $X$. The transition function $\{ p_t \}_{t\ge 0}$ and the resolvents $\{R_{\alpha} \}_{\alpha >0}$ of $\mathbb{M}$ are defined by
\begin{align}
&(p_tf)(x) \coloneqq E_{x}(f(Z_{t})), \\
&(R_{\alpha}f)(x) \coloneqq \int_{0}^{\infty}e^{-\alpha t}(p_tf)(x)\,dt, \quad x \in X, f \in B_{b}(X),
\end{align}
where $E_{x}$ denotes the expectation with respect to $P_{x}$.
\end{dfn}
\begin{prp}[cf.\ {\cite[Theorem 3.3.4]{oshima2013semi}}]\label{semidirichletassociatehunt}
    If $(\mathcal{A},\mathcal{F})$ is a regular semi-Dirichlet form on $L^2(X,\mu)$, there exists a Hunt process $\mathbb{M}$ such that $p_tf$ is quasi-continuous and $\mu$-version of $T_tf$ for any $f \in B_{b}(X) \cap L^2(X,\mu)$ and $t \ge 0$, where $\{ p_t \}_{t\ge 0}$ is the transition function of $\mathbb{M}$ and 
    $(T_t)_{t\ge 0}$ is the semigroup on $L^2(X,\mu)$ associated with $(\mathcal{A},\mathcal{F})$.
\end{prp}
In this section, from now on, we assume that $(X,d)$ is a compact metric space and 
$\mu$ is a Borel probability measure on $X$ with full support. Let
$C(X)$ denote the space of all continuous real functions on $X$ 
with the uniform norm $\norm{\cdot}_{\infty}$. For any $g \in C(X)$, let $\mathop{\mathrm{supp}} g$ denote the support of $g$.
\begin{dfn}[{\cite[Corollary 5.2.3]{oshima2013semi}}]
Let $(\mathcal{A},\mathcal{F})$ be a semi-Dirichlet form on $L^2(X, \mu)$ and assume that 
$\mathcal{F}$ is a dense subspace of $C(X)$. Then, $(\mathcal{A},\mathcal{F})$ is called strongly local 
if the following holds: For $f,g \in \mathcal{F}$, $\mathcal{A}(f,g)=0$ holds if 
there exists an open set $U \subset X$ such that 
$\mathop{\mathrm{supp}} g \subset U$ and $f$ is constant on $U$.
\end{dfn}
\begin{dfn}[{\cite[p.\ 110]{FOT}}]\label{defitionenergymeasure}
    Let $(\mathcal{E},\mathcal{F})$ be a Dirichlet form on $L^2(X, \mu)$ and assume that 
    $\mathcal{F} \subset C(X)$ and that $\mathcal{F}$ is dense in 
    $(C(X),\norm{\cdot}_{\infty})$. Then, for any $f \in \mathcal{F}$, 
    there exists a unique finite Borel measure $\nu_{f}$ on $X$ which satisfies 
    \begin{align}
        \int_{X}g\,d\nu_{f}=2\mathcal{E}(fg,f)-\mathcal{E}(g,f^{2}), 
        \quad g \in \mathcal{F}.
    \end{align}
     The measure $\nu_{f}$ is called the energy measure of $f$. Also, for any 
     $f,g \in \mathcal{F}$, we define a signed measure $\nu_{f,g}$ by 
    \begin{equation}
        \nu_{f,g} \coloneqq \frac{1}{2}(\nu_{f+g}-\nu_{f}-\nu_{g}).
    \end{equation}
    The measure $\nu_{f,g}$ is called the mutual energy measure of $f$ and $g$. Then, we have 
    $\abs{\nu_{f,g}(B)} \leq \nu_{f}(B)^{1/2}\nu_{g}(B)^{1/2}$
    for any Borel set $B$ of $X$.
\end{dfn}
\begin{rem}
If $1 \in \mathcal{F}$ and $\mathcal{E}(1)=0$ in addition, then $\nu_{f}(X)=2\mathcal{E}(f)$ for any $f \in \mathcal{F}$.
\end{rem}
The next lemma is straightforward and we omit the proof.
\begin{lem}\label{enerzero}
Let $(\mathcal{E},\mathcal{F})$ be a Dirichlet form on $L^2(X, \mu)$ and assume that 
    $1 \in \mathcal{F} \subset C(X)$, $\mathcal{F}$ is dense in 
    $(C(X),\norm{\cdot}_{\infty})$ and $(\mathcal{E},\mathcal{F})$ is strongly local. 
    Then, $\nu_{f}(G)=0$ holds for any $f \in \mathcal{F}$ and an open set $G \subset X$ 
    such that $f$ is constant on $G$.
\end{lem}
\subsection{Resistance forms}
We mainly follow \cite{kigami2001analysis} to introduce the resistance form. Let $X$ be a set in this subsection. For any set $A$, $l(A)$ denotes the set of all real functions on $A$.
\begin{dfn}[{\cite[Definition 2.3.1]{kigami2001analysis}}]
Let $\mathcal{F}$ be a subspace of $l(X)$ containing all the constant functions on 
$X$, and $\mathcal{E} \colon \mathcal{F} \times \mathcal{F} \rightarrow \R$ a bilinear form. 
Then, $(\mathcal{E},\mathcal{F})$ is called a resistance form on $X$ if the following hold:
\begin{enumerate}
\item[$\mathrm{(RF1)}$]
$\mathcal{E}(g,f)=\mathcal{E}(f,g)$ and $\mathcal{E}(f,f) \geq 0$ for $f,g \in \mathcal{F}$. 
Also, $\mathcal{E}(f,f)=0$ if and only if $f$ is a constant function.
\item[$\mathrm{(RF2)}$]
$\left( \mathcal{F}/ {\sim}, \mathcal{E} \right)$ is a Hilbert space, where 
$\sim$ is an equivalence relation on $\mathcal{F}$ defined by 
$f \sim g$ if $f-g$ is a constant function.
\item[$\mathrm{(RF3)}$]
For any finite subset $V \subset X$ and $f^{\prime} \in l(V)$, there exists some
$f \in \mathcal{F}$ such that $f \lvert_{V} =f^{\prime}$ holds.
\item[$\mathrm{(RF4)}$]
For any $x,y \in X$, we have 
\begin{align}
    R(x,y) \coloneqq \sup \left\{\frac{\abs{f(x)-f(y)}^{2}}{\mathcal{E}(f)} \middle |
    f \in \mathcal{F},f \, \text{is not constant}\right\} < \infty.
\end{align}
\item[$\mathrm{(RF5)}$]
For any $f \in \mathcal{F}$, we have $(0 \vee f) \wedge 1 \in \mathcal{F}$ and 
$\mathcal{E}((0 \vee f) \wedge 1)\leq \mathcal{E}(f)$.
\end{enumerate}
Then, $R$ in $\mathrm{(RF4)}$ is a metric on $X$, called the resistance metric associated 
with $(\mathcal{E},\mathcal{F})$.

In what follows, $(\mathcal{E},\mathcal{F})$ is a resistance form on $X$ and $V$ is a non-empty finite subset of $X$.
\end{dfn}
\begin{prp}[{\cite[Lemma 2.3.5 and Theorem 2.3.4]{kigami2001analysis}}]\label{resistancetrace}
For any $f^{\prime} \in l(V)$, there exists a unique 
$\widehat{f^{\prime}} \in \mathcal{F}$ such that 
$\widehat{f^{\prime}}\lvert_{V}=f^{\prime}$ and 
$\mathcal{E}(\widehat{f^{\prime}})=
\min \{ \mathcal{E}(f) \mid f \in \mathcal{F}, f \lvert_{V}=f^{\prime} \}$. 
By defining $\mathcal{E}^{\prime} \colon l(V) \times l(V) \rightarrow \R$ by 
\begin{equation}
\mathcal{E}^{\prime}(f^{\prime},g^{\prime})
\coloneqq
\mathcal{E}(\widehat{f^{\prime}},\widehat{g^{\prime}}),
\quad f^{\prime},g^{\prime} \in l(V),
\end{equation}
$(\mathcal{E}^{\prime},l(V))$ is a resistance form on $V$. 
Also, $R^{\prime}=R \lvert_{V \times V}$ holds, where 
$R$ $($resp.\ $R^{\prime}$$)$ is the resistance metric on $X$ $($resp.\ $V$$)$ associated with 
$(\mathcal{E},\mathcal{F})$ $($resp.\ $(\mathcal{E}^{\prime},l(V))$$)$.
\end{prp}
\begin{dfn}[{\cite[Definition 8.3]{MR2919892}}]
In the notation of Proposition~\ref{resistancetrace}, $\widehat{f^{\prime}} \in \mathcal{F}$ is called the harmonic extension of $f^{\prime} \in l(V)$ 
with respect to $(\mathcal{E},\mathcal{F})$ and 
$(\mathcal{E}^{\prime},l(V))$ is called the trace of $(\mathcal{E},\mathcal{F})$ onto $V$ and usually denoted by $\mathop{\mathrm{Tr}}(\mathcal{E} \lvert V)$.
\end{dfn}
\begin{rem}\label{traceineq}
    Note that, for any $f \in \mathcal{F}$, 
    $\mathop{\mathrm{Tr}}(\mathcal{E} \lvert V)(f\lvert_{V}) \le \mathcal{E}(f)$ holds.
\end{rem}

For any subset $A$ of $X$, let $\mathrm{diam}_{R}A$ denote the diameter of $A$ with respect to the metric $R$. Let $\mu$ be a Borel probability measure on $X$ with full support and assume that $(X,R)$ is compact.  Since $\mathcal{F}$ is an algebra (from e.g.\ \cite[Lemma~6.5]{MR2919892}), contains the constant function $1$, and has the point-separating property from $\mathrm{(RF3)}$, $\mathcal{F}$ is dense in 
$(C(X),\norm{\cdot}_{\infty})$ by the Stone--Weierstrass theorem. From 
\cite[p.\ 65]{kigami2001analysis}, $(\mathcal{E},\mathcal{F})$ is a 
Dirichlet form on $L^2(X,\mu)$. 
Also, the following propositions hold. These are standard and well-known results in the theory of resistance forms and we omit their proofs.
\begin{prp}\label{kihon}
For any $f \in \mathcal{F}$, 
\begin{equation}
\norm{f}_{\infty} \leq
\left(\mathrm{diam}_{R}X\right)^{1/2} \mathcal{E}(f)^{1/2}+\norm{f}_{L^{2}(X,\mu)}
\end{equation}
holds. In particular, we have 
\begin{equation}
    \norm{f}_{\infty} 
    \leq\left( \left(\mathrm{diam}_{R}X\right)^{1/2}+1 \right) \mathcal{E}_{1}(f)^{1/2}.
\end{equation}
\end{prp}
\begin{prp}\label{cptop}
The embedding $i_{\mathcal{F}} \colon \mathcal{F} \rightarrow C(X)$ 
from the Hilbert space $(\mathcal{F}, \mathcal{E}_{1}(\cdot))$ to the Banach space 
$\left(C(X),\norm{\cdot}_{\infty} \right)$ is a compact operator.
\end{prp}
\section{Problem settings}
\subsection{Basic settings}
We recall the problem setting as mentioned in Section 1. 
Let $X$ be a set and $(\mathcal{E},\mathcal{F})$ a resistance form on $X$. 
Assume that the metric space $(X,R)$ is compact, where $R$ denotes the 
resistance metric associated with $(\mathcal{E},\mathcal{F})$. Let $\mu$ be 
a Borel probability measure on $X$ with full support and assume that 
$(\mathcal{E},\mathcal{F})$ is a strongly local Dirichlet form on 
$L^2(X,\mu)$. For $f,g \in \mathcal{F}$, $\nu_{f,g}$ denotes the mutual 
energy measure of $f$ and $g$ as in Definition~\ref{defitionenergymeasure}. Fix a natural number $N$, 
$b_{1},\dots,b_{N} \in C(X)$ and $h_{1},\dots,h_{N} \in \mathcal{F}$. Let
$\{ V_{n} \}_{n \in \N}$ be a non-decreasing sequence of nonempty finite 
subsets of $X$ and assume that $V_{*} \eqdot \bigcup_{n=1}^{\infty} V_{n}$ is dense 
in $(X,R)$. For each $n \in \N$, we define $\mathcal{E}^{n}$, a resistance form on $V_{n}$, 
by $\mathcal{E}^{n} \eqdot \mathrm{Tr}(\mathcal{E} \lvert V_{n})$. Let $\mu_{n}$ 
be a probability measure on $V_{n}$ which satisfies $\mu_{n}( \{ x\}) >0$ for all 
$x \in V_{n}$ and assume that $\{ \mu_{n} \}_{n \in \N}$ converges weakly to $\mu$. 
For each $1 \leq i \leq N$, we define 
$\mathcal{Q}_{i} \colon \mathcal{F} \times \mathcal{F} \rightarrow \R$ by
\begin{equation}
\mathcal{Q}_{i}(f,g) \eqdot \frac{1}{2}\int_{X} b_{i}g \,d\nu_{f,h_{i}}
, \quad f,g \in \mathcal{F},
\end{equation}
and we define 
$\mathcal{Q}, \mathcal{A} \colon \mathcal{F} \times \mathcal{F} \rightarrow \R$ by
\begin{align}
\mathcal{Q} \eqdot \sum_{i=1}^{N}\mathcal{Q}_{i} \text{ and }
\mathcal{A} \eqdot \mathcal{E}+\mathcal{Q}.
\end{align}
For each $n \in \N$, $\{c_{n,x,y} \}_{\substack{x ,y \in V_{n} \\x \neq y}} \subset \nonnegaR$ is uniquely defined by 
\begin{align}\label{coefofen}
&\mathcal{E}^{n}(f,g)= 
\frac{1}{2}
\sum_{
\substack{
x ,y \in V_{n} \\
x \neq y}
}
c_{n,x,y}(f(x)-f(y))(g(x)-g(y)),\\
&\quad f,g \in l(V_{n}),\,
 c_{n,y,x}=c_{n,x,y}, \quad x ,y \in V_{n} ,x \neq y.
\end{align}
For each $1 \leq i \leq N$, we define 
$\mathcal{Q}^{n}_{i} \colon l(V_{n}) \times l(V_{n}) \rightarrow \R$ by 
\begin{align}
 &\mathcal{Q}^{n}_{i}(f,g) \eqdot
 \frac{1}{2}
 \sum_{
\substack{
x ,y \in V_{n} \\
x \neq y}
}
c_{n,x,y}b_{i}(x)g(x)(f(x)-f(y))(h_{i}(x)-h_{i}(y)), \\
&\quad f,g \in l(V_{n}).
\end{align} 
Also, we define 
$\mathcal{Q}^{n},\mathcal{A}^{n} \colon l(V_{n}) \times l(V_{n}) \rightarrow \R$ by 
\begin{align}
\mathcal{Q}^{n} \eqdot \sum_{i=1}^{N}\mathcal{Q}^{n}_{i}, \,
\mathcal{A}^{n}:=\mathcal{E}^{n}+\mathcal{Q}^{n}.
\end{align}
Let $R_{n}$ be the resistance metric on $V_{n}$ associated with $\mathcal{E}^{n}$. Note that we have 
$R_{n}=R\lvert_{V_{n}\times V_{n}}$. Also, for $f \in l(V_{n})$, 
$\widehat{f} \in \mathcal{F}$ denotes the harmonic extension of $f$ on $X$ 
with respect to $(\mathcal{E},\mathcal{F})$. 
\begin{rem}\label{entraceineq}
    For any $n \in \N$, we have $\mathcal{E}^{n}(f\lvert_{V_{n}}) \le \mathcal{E}(f)$
    for any $f \in \mathcal{F}$ from Remark~\ref{traceineq} since $\mathcal{E}^{n}= \mathrm{Tr}(\mathcal{E} \lvert V_{n})$ holds.
\end{rem}
First, note that the following elementary lemma holds since $\{ \mu_{n}\}_{n \in \N}$ weakly converges to $\mu$.
\begin{lem} \label{easy}
Let $\{ f^{(n)}\}_{n \in \N}\subset C(X)$ and $f\in C(X)$. If $\limtoinfty{n}\norm{f^{(n)}-f}_{\infty}=0$,
then 
$\limtoinfty{n}
\int_{V_{n}}f^{(n)} \lvert_{V_{n}}\,d\mu_{n}
=\int_{X}f\,d\mu$.
\end{lem}
From Lemma~\ref{easy}, we immediately have the following: 
\begin{prp} \label{KS1}
For $n \in \N$, we define $\Phi_{n} \colon C(X) \rightarrow L^{2}(V_n, \mu_{n})$ by 
$\Phi_{n}(f)=f\lvert_{V_n}$. Then, $\{L^{2}(V_n, \mu_{n}) \}_{n \in \N}$ converges to $L^{2}(X, \mu)$ 
in the sense of Definition~\ref{ksspaceconve}.
\end{prp}
For the rest of this section, we consider Kuwae--Shioya's framework defined by
Proposition~\ref{KS1} unless otherwise specified.
\begin{prp}\label{prp1}
Let $\{ f^{(n)}\}_{n \in \N}\subset C(X)$ and $f\in C(X)$ and assume that $\limtoinfty{n}\norm{f^{(n)}-f}_{\infty}=0$. Then 
$f^{(n)} \lvert_{V_{n}} \in L^{2}(V_{n}, \mu_{n})$ 
strongly K-S converges to $f \in L^{2}(X, \mu)$.
\end{prp}
\begin{proof}
We have $\limtoinfty{n}\norm{(f^{(n)})^{2}-f^{2}}_{\infty}=0$ from 
$\limtoinfty{n}\norm{f^{(n)}-f}_{\infty}=0$. Thus, by Lemma~\ref{easy}, we 
also have 
$\limtoinfty{n}\norm{f^{(n)} \lvert_{V_{n}}}_{L^{2}(V_{n}, \mu_{n})} 
=\norm{f}_{L^{2}(X, \mu)}$. 

Next, for an arbitrary $g \in C(X)$, we have 
$\limtoinfty{n}\norm{f^{(n)}g-fg}_{\infty}=0$ from 
$\limtoinfty{n}\norm{f^{(n)}-f}_{\infty}=0$. Thus, again by Lemma~\ref{easy}, 
we have
\begin{align}
\limtoinfty{n}(f^{(n)} \lvert_{V_{n}},g\lvert_{V_{n}})_{L^{2}(V_{n}, \mu_{n})} 
=(f,g)_{L^{2}(X, \mu)}.
\end{align}
By Proposition~\ref{kyo1}~\ref{enum:8kyo}, the conclusion holds.
\end{proof}
From Propositions~\ref{prp1} and \ref{cptop}, we have the following:
\begin{cor}\label{cor1}
If $\{ f^{(n)}\}_{n \in \N} \subset \mathcal{F}$ weakly converges to 
$f \in \mathcal{F}$ in the Hilbert space $(\mathcal{F}, \mathcal{E}_{1}(\cdot))$, 
then $\{f^{(n)}\lvert_{V_{n}} \}_{n \in \N} 
\,\left( f^{(n)} \lvert_{V_{n}} \in L^{2}(V_n, \mu_{n}) \right)$ 
strongly K-S converges to $f \in L^{2}(X, \mu)$.
\end{cor}
\begin{prp}\label{prp2}
Let $f_{n} \in L^{2}(V_n, \mu_{n})$ for $n \in \N$ and 
assume $\supover{n}\mathcal{E}^{n}_{1}(f_{n}) < \infty$. Then, there 
exist a subsequence $\{ f_{n_{k}} \}_{k \in \N}$ of $\{ f_{n} \}_{n \in \N}$ 
and $f \in  \mathcal{F}$ such that
$\{ \widehat{f_{n_{k}}} \}_{k \in \N}$ weakly converges to $f$ 
in the Hilbert space $(\mathcal{F}, \mathcal{E}_{1}(\cdot))$.
\end{prp}
\begin{proof}
For $n \in \N$, we have 
$\norm{\widehat{f_{n}}}_{\infty}=\norm{f_{n}}_{\infty}$ from the maximum principle for $V_{n}$-harmonic functions of $(\mathcal{E}, \mathcal{F})$ (see \cite[Proposition~2.18]{kigami2003harmonic}). Thus we obtain
\begin{align} 
\norm{\widehat{f_{n}}}^{2}_{L^2(X,\mu)}
&\leq
\norm{\widehat{f_{n}}}^{2}_{\infty}
=\norm{f_{n}}^{2}_{\infty} \\
&\leq
\left( \left(\mathrm{diam}_{R_{n}}V_{n} \right)^{1/2}+1 \right)^{2}
 \mathcal{E}^{n}_{1}(f_{n}) \\
&\leq
\left( \left(\mathrm{diam}_{R}X \right)^{1/2}+1 \right)^{2} 
\supover{m}\mathcal{E}^{m}_{1}(f_{m})
\end{align}
from Proposition \ref{kihon}. Thus we have
\begin{align}
\mathcal{E}_{1}(\widehat{f_{n}}) 
&=\mathcal{E}(\widehat{f_{n}})+\norm{\widehat{f_{n}}}^{2}_{L^2(X,\mu)} \\
&=\mathcal{E}^{n}(f_{n})+\norm{\widehat{f_{n}}}^{2}_{L^2(X,\mu)} \\
&\leq
\left(\left( \left(\mathrm{diam}_{R}X \right)^{1/2}+1 \right)^{2}+1 \right)
\supover{m}\mathcal{E}^{m}_{1}(f_{m}) < \infty.
\end{align}
Thus, $\{ \widehat{f_{n}} \}_{n \in \N}$ is bounded in $(\mathcal{F}, \mathcal{E}_{1}(\cdot))$. The conclusion follows from 
the Banach--Alaoglu theorem.
\end{proof}
\begin{cor}\label{cor2}
Let $f_{n} \in L^{2}(V_n, \mu_{n})$ for $n \in \N$ and assume that 
$\{ f_{n} \}_{n \in \N}$ weakly K-S converges to $g \in L^2(X, \mu)$ and $\supover{n}\mathcal{E}^{n}_{1}(f_{n}) < \infty$. Then, $g \in \mathcal{F}$ 
and $\{ \widehat{f_{n}} \}_{n \in \N}$ weakly converges 
to $g$ in the Hilbert space $(\mathcal{F}, \mathcal{E}_{1}(\cdot))$.
\end{cor}
\begin{proof}
We take an arbitrary subsequence $\{ \widehat{f_{n_{k}}} \}_{k \in \N}$
of $\{ \widehat{f_{n}} \}_{n \in \N}$. Then, there exists 
a subsequence $\{ f_{n_{k_{l}}} \}_{l \in \N}$ of $\{ f_{n_{k}} \}_{k \in \N}$
and $f \in \mathcal{F}$ such that $\{ \widehat{f_{n_{k_{l}}}} \}_{l \in \N}$
weakly converges to $f$ in $(\mathcal{F}, \mathcal{E}_{1}(\cdot))$
by Proposition~\ref{prp2} and $\supover{k}\mathcal{E}^{n_{k}}_{1}(f_{n_{k}}) < \infty$. By Corollary~\ref{cor1}, $\{ f_{n_{k_{l}}} \}_{l \in \N}
=\{ \widehat{f_{n_{k_{l}}}}\lvert_{V_{n_{k_{l}}}} \}_{l \in \N}$ 
strongly K-S converges to $f$. On the other hand, 
$\{ f_{n_{k_{l}}} \}_{l \in \N}$ weakly K-S converges to $g$ by 
Proposition~\ref{wk1}~\ref{enum:4wk}. Thus we have $f=g$ by Proposition~\ref{wk1}~\ref{enum:5wk}. In particular, $g \in \mathcal{F}$ holds and 
$\{ \widehat{f_{n_{k_{l}}}} \}_{l \in \N}$ weakly converges to $g$
in $(\mathcal{F}, \mathcal{E}_{1}(\cdot))$. By usual 
subsequence arguments, the conclusion holds.
\end{proof}
\begin{prp}\label{prp3}
Let $f_{n}  \in L^{2}(V_n, \mu_{n})$ for $n \in \N$ and assume that 
$\{ f_{n}\}_{n \in \N}$ weakly K-S converges to $f \in L^{2}(X, \mu)$. Then, for 
any $g \in C(X)$, $\{ f_{n}g \lvert_{V_{n}}\}_{n \in \N}$ weakly K-S converges 
to $fg \in L^{2}(X, \mu)$.
\end{prp}
\begin{proof}
$\supover{n}\norm{f_{n}}_{L^{2}(V_n, \mu_{n})} < \infty$ holds since 
$\{ f_{n}\}_{n \in \N}$ weakly K-S converges to $f \in L^{2}(X, \mu)$.
Thus, $\supover{n}\norm{f_{n}g \lvert_{V_{n}}}_{L^{2}(V_n, \mu_{n})}
\leq\norm{g}_{\infty} \supover{n}\norm{f_{n}}_{L^{2}(V_n, \mu_{n})}
< \infty$ holds. Next, we take an arbitrary $h \in C(X)$. Then, from Proposition~\ref{wk1}~\ref{enum:1wk}, we have
\begin{align}
\limtoinfty{n}(f_{n}g \lvert_{V_{n}},h\lvert_{V_{n}})_{L^{2}(V_n, \mu_{n})}
&=\limtoinfty{n}(f_{n},(gh) \lvert_{V_{n}})_{L^{2}(V_n, \mu_{n})} \\
&=(f,gh)_{ L^{2}(X, \mu)}
=(fg,h)_{ L^{2}(X, \mu)}
\end{align}
since $\{ f_{n}\}_{n \in \N}$ weakly K-S converges to $f \in L^{2}(X, \mu)$ 
and $gh \in C(X)$. Thus, by Proposition~\ref{wk1}~\ref{enum:1wk} again, 
the conclusion holds.
\end{proof}
\begin{prp}\label{seki}
Let $f_{n}  \in L^{2}(V_n, \mu_{n})$ for $n \in \N$ and assume that 
$\{ f_{n}\}_{n \in \N}$ weakly K-S converges to $f \in L^{2}(X, \mu)$ and 
$\supover{n}\mathcal{E}^{n}_{1}(f_{n}) < \infty$ holds. Then, for any 
$g \in \mathcal{F}$, $\{ \widehat{f_{n}g \lvert_{V_{n}}} \}_{n \in \N}$ weakly 
converges to $fg$ in the Hilbert space $(\mathcal{F}, \mathcal{E}_{1}(\cdot))$.
\end{prp}
\begin{proof}
First, note that $\supover{n}\norm{f_{n}}_{\infty} < \infty$ holds by 
$\supover{n}\mathcal{E}^{n}_{1}(f_{n}) < \infty$ and Proposition \ref{kihon}.
Thus, for all $n \in \N$, we have 
\begin{align}
\norm{f_{n}g \lvert_{V_{n}}}_{L^{2}(V_n, \mu_{n})}
\leq \norm{f_{n}g \lvert_{V_{n}}}_{\infty}
\leq \norm{g}_{\infty} \supover{k}\norm{f_{k}}_{\infty} < \infty
\end{align}
and 
\begin{align}
\mathcal{E}^{n}(f_{n}g \lvert_{V_{n}})^{1/2}
&\leq \norm{g \lvert_{V_{n}}}_{\infty} \mathcal{E}^{n}(f_{n})^{1/2}
+\norm{f_{n}}_{\infty} \mathcal{E}^{n}(g \lvert_{V_{n}})^{1/2} \\
&\leq\norm{g}_{\infty} \supover{k}\mathcal{E}^{k}_{1}(f_{k})^{1/2}
+\supover{k}\norm{f_{k}}_{\infty} \mathcal{E}(g)^{1/2} \\
&< \infty.
\end{align}
Thus we obtain $\supover{n}\mathcal{E}^{n}_{1}(f_{n}g \lvert_{V_{n}}) < \infty$.
On the other hand, since $\{ f_{n}\}_{n \in \N}$ weakly K-S converges to 
$f$, $\{ f_{n}g \lvert_{V_{n}}\}_{n \in \N}$ weakly K-S converges to $fg$ 
by Proposition~\ref{prp3}. Thus the conclusion holds by Corollary~\ref{cor2}.
\end{proof}
Since $V_{*} = \bigcup_{n=1}^{\infty} V_{n}$ is dense in $(X,R)$, we have the following:
\begin{prp}[cf.\ {\cite[Theorem 3.14]{MR2919892}}]\label{tracelimit}
We obtain 
\begin{align}
&\mathcal{F}=\{ f \in C(X) \mid \sup_{n \in \N}\mathcal{E}^{n}(f \lvert_{V_{n}}) = \lim_{n \to \infty}\mathcal{E}^{n}(f \lvert_{V_{n}}) < \infty\}, \text{ and} \\
&\mathcal{E}(f,g)=\lim_{n \to \infty}
\mathcal{E}^{n}(f \lvert_{V_{n}}, g \lvert_{V_{n}}), \quad f,g \in \mathcal{F}.
\end{align}
\end{prp}
\begin{lem}\label{techn1}
$\limtoinfty{n}\mathcal{E}(g-\widehat{g \lvert_{V_{n}}})=0$ holds for any $g \in \mathcal{F}$.
\end{lem}
\begin{proof}
For $n \in \N$, since $\widehat{g \lvert_{V_{n}}}$ is the harmonic extension of $g \lvert_{V_{n}}$, $\mathcal{E}(h,\widehat{g \lvert_{V_{n}}})=0$ 
holds for any $h \in \mathcal{F}$ which satisfies $h \lvert_{V_{n}}=0$. In particular,
$\mathcal{E}(g-\widehat{g \lvert_{V_{n}}},\widehat{g \lvert_{V_{n}}})=0$ holds. 
Thus we obtain
\begin{align}
\mathcal{E}(g)&=\mathcal{E}((g-\widehat{g \lvert_{V_{n}}})+\widehat{g \lvert_{V_{n}}})\\
&=\mathcal{E}(g-\widehat{g \lvert_{V_{n}}})
+\mathcal{E}(\widehat{g \lvert_{V_{n}}})
+2\mathcal{E}(g-\widehat{g \lvert_{V_{n}}},\widehat{g \lvert_{V_{n}}})\\
&=\mathcal{E}(g-\widehat{g \lvert_{V_{n}}})+\mathcal{E}^{n}(g \lvert_{V_{n}}).
\end{align}
On the other hand, from Proposition~\ref{tracelimit}, we have $\mathcal{E}(g)=\lim_{n \to \infty}\mathcal{E}^{n}(g \lvert_{V_{n}})$. Therefore, 
$\limtoinfty{n}\mathcal{E}(g-\widehat{g \lvert_{V_{n}}})=0$ holds.
\end{proof}
\subsection{Additional assumptions}
We consider the following Conditions \ref{assmp1}, \ref{assmp2}, and \ref{assmp3}:
\begin{enumerate}[label=\textbf{(\Roman*)}, ref=\bf{(\Roman*)}]
    \item\label{assmp1} 
    $\sum_{i,j=1}^{N} \int_{X} b_{i}b_{j} \, d\nu_{h_{i}, h_{j}}< 
        2\left(\mathrm{diam}_{R}X\right)^{-1}.$
    \item\label{assmp2} 
    For any $x \in X$, we have 
    \begin{align}
        \mathcal{E}\left(\sum_{i=1}^{N}b_{i}(x)h_{i}\right) \le \left(\mathrm{diam}_{R}X \right)^{-1}.
    \end{align}
    \item\label{assmp3}
        For any $n \in \N$, $X \setminus V_{n}$ has finitely many connected components
    $\{U_{\lambda} \}_{\lambda \in \Lambda_{n}}$, and for any $x \in X$, 
    $\left\{ \bigcup_{\substack{\lambda \in \Lambda_{n}, \\ 
    x \in \mathrm{Cl}(U_{\lambda})}} \mathrm{Cl}(U_{\lambda}) \right\}_{n \in \N}$ 
    is a fundamental system of neighbourhoods of $x$, where $\mathrm{Cl}(U_{\lambda})$ 
    denotes the closure of $U_{\lambda}$ in $(X,R)$.
\end{enumerate}
Condition~\ref{assmp1} is a smallness condition for the drift part. 
Condition~\ref{assmp2} is a pointwise smallness condition which implies 
the Markov property of the approximating forms. 
Condition~\ref{assmp3} means a kind of finite ramification. We refer to Conditions~\ref{assmp1} and \ref{assmp2} as Assumption A, and Conditions~\ref{assmp1} and \ref{assmp3} as Assumption B. From now on, we assume Assumption A or B. 

For any $n \in \N$ and $\lambda \in \Lambda_{n}$, let $\partial U_{\lambda}$ denote the boundary of $U_{\lambda}$ in $(X, R)$. Since $(\mathcal{E},\mathcal{F})$ is strongly local, we have the following lemma:
\begin{lem}[{\cite[Lemmas 3.2, 3.4, 3.9]{hino2014geodesic}}]
Under Condition \ref{assmp3}, the following hold.
\begin{enumerate}[label=\normalfont(\roman*)]\label{cond3first}
    \item \label{enum:1condition3}
    $\limtoinfty{n} \supover{\lambda \in \Lambda_{n}} \mathrm{diam}_{R} \mathrm{Cl}(U_{\lambda})=0$.
    \item \label{enum:2condition3}
    For any $n \in \N$ and $\lambda \in \Lambda_{n}$, $\partial U_{\lambda} \subset V_{n}$.
    \item \label{enum:3condition3}
    For any $n \in \N$ and $\lambda \in \Lambda_{n}$, there exists $\{a^{\lambda}_{n,p,q} \}_{\substack{p ,q \in \partial U_{\lambda} \\p \neq q}} \subset \nonnegaR$ such that for all distinct $p ,q \in \partial U_{\lambda}$, $a^{\lambda}_{n,p,q}=a^{\lambda}_{n,q,p}$ and the following equation holds:
    \begin{align}
    \nu_{\widehat{f}}(U_{\lambda})=\sum_{\substack{p ,q \in \partial U_{\lambda} \\p \neq q}}
    a^{\lambda}_{n,p,q}(f(p)-f(q))^{2}, \quad f \in l(V_{n}).
    \end{align}
    \end{enumerate}
\end{lem}
\begin{lem}\label{localcoef}
Under Condition \ref{assmp3}, the following holds:
\begin{align}
    c_{n,x,y}=
        \sum_{\lambda \in \Lambda_{n}} 1_{\partial U_{\lambda}}(x)1_{\partial U_{\lambda}}(y)
        a^{\lambda}_{n,x,y}, \quad n \in \N, x ,y \in V_{n} ,x \neq y,
\end{align}
where $\{a^{\lambda}_{n,p,q} \}_{\substack{p ,q \in \partial U_{\lambda} \\p \neq q}} \subset \nonnegaR$ are the same as in Lemma~\ref{cond3first}~\ref{enum:3condition3}.
\end{lem}
\begin{proof}
We take an arbitrary $f \in l(V_{n})$. Since $\widehat{f}(V_{n})$ is a finite subset of $\R$, $\nu_{\widehat{f}}(\widehat{f}^{-1}(\widehat{f}(V_{n})))=0$ (from e.g.\ \cite[Theorem~4.3.8]{chen2012symmetric}), which implies $\nu_{\widehat{f}}(V_{n})=0$. Thus, we have 
\begin{align}
    \mathcal{E}^{n}(f)
    &=\mathcal{E}(\widehat{f}) \\
    &= \frac{1}{2} \nu_{\widehat{f}}(X) \\
    &= \frac{1}{2} \sum_{\lambda \in \Lambda_{n}} \nu_{\widehat{f}}(U_{\lambda})
    +\frac{1}{2} \nu_{\widehat{f}}(V_{n}) \\
    &= \frac{1}{2} \sum_{\lambda \in \Lambda_{n}} 
    \sum_{\substack{p ,q \in \partial U_{\lambda} \\p \neq q}}
    a^{\lambda}_{n,p,q}(f(p)-f(q))^{2} +0 \\
    &=\frac{1}{2}\sum_{\substack{x ,y \in V_{n} \\x \neq y}}
    \left( 
        \sum_{\lambda \in \Lambda_{n}} 1_{\partial U_{\lambda}}(x)1_{\partial U_{\lambda}}(y)
        a^{\lambda}_{n,x,y}
    \right)
    (f(x)-f(y))^2.
\end{align}
By the uniqueness of the representation in \eqref{coefofen}, we obtain 
the conclusion.
\end{proof}
\begin{lem}\label{importantb}
Let $g \in C(X), h, h^{\prime} \in \mathcal{F}$. Then, we have
\begin{align}
    \limtoinfty{n}\sum_{\substack{x ,y \in V_{n} \\x \neq y}}
    c_{n,x,y}g(x)(h(x)-h(y))(h^{\prime}(x)-h^{\prime}(y))
    = \int_{X} g\, d\nu_{h, h^{\prime}}.
\end{align}
\end{lem}
\begin{proof}
A direct computation gives
\begin{align}
\sum_{\substack{x ,y \in V_{n} \\x \neq y}}c_{n,x,y}g(x)(h(x)-h(y))^{2}
=2\mathcal{E}^{n}((hg)\lvert_{V_{n}},h\lvert_{V_{n}})
-\mathcal{E}^{n}(g\lvert_{V_{n}},h^2\lvert_{V_{n}})
\end{align}
for any $n \in \N$ and $g,h \in \mathcal{F}$. Thus, by taking the limit, we have, from Proposition~\ref{tracelimit}, 
\begin{align}
\limtoinfty{n}\sum_{\substack{x ,y \in V_{n} \\x \neq y}}c_{n,x,y}g(x)(h(x)-h(y))^{2}
&=2\mathcal{E}(hg,h)-\mathcal{E}(g,h^2)\\
&=\int_{X} g \,d\nu_{h}, \quad g,h \in \mathcal{F}.
\end{align}
Next, we take arbitrary $g \in C(X)$, $h \in \mathcal{F}$, and $\varepsilon >0$. Since $\mathcal{F}$ is dense in $(C(X),\norm{\cdot}_{\infty})$, there exists $g^{\prime} \in \mathcal{F}$ such that 
$\norm{g-g^{\prime}}_{\infty} < \varepsilon$. Then, for any $n \in \N$, we have 
\begin{align}
    &\abs{\sum_{\substack{x ,y \in V_{n} \\x \neq y}}c_{n,x,y}g(x)(h(x)-h(y))^{2}
    -\sum_{\substack{x ,y \in V_{n} \\x \neq y}}c_{n,x,y}g^{\prime}(x)(h(x)-h(y))^{2}}\\
    &=\abs{\sum_{\substack{x ,y \in V_{n} \\x \neq y}}c_{n,x,y}(g(x)-g^{\prime}(x))(h(x)-h(y))^{2}} \\
    &\le \sum_{\substack{x ,y \in V_{n} \\x \neq y}}c_{n,x,y}\abs{g(x)-g^{\prime}(x)}(h(x)-h(y))^{2} \\ 
    &\le \norm{g-g^{\prime}}_{\infty}\sum_{\substack{x ,y \in V_{n} \\x \neq y}}c_{n,x,y}(h(x)-h(y))^{2} \\
    &\le 2\varepsilon \mathcal{E}^{n}(h\lvert_{V_{n}}) \\
    &\le 2\varepsilon\mathcal{E}(h)
\end{align}
and
\begin{align}
    \abs{\int_{X} g^{\prime} \,d\nu_{h}-\int_{X} g \,d\nu_{h}}
    &=\abs{\int_{X} (g^{\prime}-g) \,d\nu_{h}} \\
    &\le \int_{X} \norm{g-g^{\prime}}_{\infty} \,d\nu_{h} \\
    &\le 2\varepsilon \mathcal{E}(h).
\end{align}
Thus, from $g^{\prime}\in \mathcal{F}$, we obtain
\begin{align}
    &\limsuptoinfty{n}\abs{\sum_{\substack{x ,y \in V_{n} \\x \neq y}}c_{n,x,y}g(x)(h(x)-h(y))^{2}
    -\int_{X} g \,d\nu_{h}} \\
    &\le
    \limsuptoinfty{n}\abs{\sum_{\substack{x ,y \in V_{n} \\x \neq y}}c_{n,x,y}
    g^{\prime}(x)(h(x)-h(y))^{2}
    -\int_{X} g^{\prime} \,d\nu_{h}}+4\varepsilon \mathcal{E}(h) \\
    &=4\varepsilon \mathcal{E}(h) \longrightarrow0 \quad 
    \text{(}\varepsilon \downarrow0\text{)}.
\end{align}
Therefore, for any $g \in C(X)$ and $h \in \mathcal{F}$, we have
\begin{align}\label{approxxxxi}
\limtoinfty{n}\sum_{\substack{x ,y \in V_{n} \\x \neq y}}c_{n,x,y}g(x)(h(x)-h(y))^{2}
=\int_{X} g \,d\nu_{h}.
\end{align}
We easily obtain the conclusion from \eqref{approxxxxi}.
\end{proof}
Note that, from Lemma~\ref{importantb}, we obtain the following:
\begin{align}
    \sum_{i,j=1}^{N} \int_{X} b_{i}b_{j} \, d\nu_{h_{i}, h_{j}}
    =\limtoinfty{n}\sum_{\substack{x ,y \in V_{n} \\x \neq y}}c_{n,x,y}\left(\sum_{i=1}^{N}b_{i}(x)(h_{i}(x)-h_{i}(y))\right)^{2} \ge0.
\end{align}
Thus, we fix some $\delta >0$ and $s>0$ which satisfy
\begin{align}
\left(\frac{1}{2}\sum_{i,j=1}^{N} \int_{X} b_{i}b_{j} \, d\nu_{h_{i}, h_{j}} 
    \right)^{1/2}
\left( \mathrm{diam}_{R}X^{1/2}+\delta \right)
 < s < 1
\end{align}
and define $\lambda$ and $t$ by 
\begin{align}
&\lambda \eqdot (4 \delta)^{-1}
\left(\mathrm{diam}_{R}X^{1/2}+\delta \right)^{-1} >0, \\
&t \eqdot \lambda s  > (4 \delta)^{-1}
\left(\frac{1}{2}\sum_{i,j=1}^{N} \int_{X} b_{i}b_{j} \, d\nu_{h_{i}, h_{j}} 
    \right)^{1/2}.
\end{align}
\begin{lem}\label{daigilem}
    For large enough $n \in \N$, we have
\begin{align}
\abs{\mathcal{Q}^{n}(f)}
\leq 
s\mathcal{E}^{n}(f)+t\norm{f}^{2}_{L^{2}(V_n, \mu_{n})}, 
\quad f\in l(V_{n}).
\end{align}
\end{lem}
\begin{proof}
For any $f\in l(V_{n})$, we have 
\begin{align}
\abs{\mathcal{Q}^{n}(f)}
&=\abs{
\frac{1}{2}\sum_{\substack{x ,y \in V_{n} \\x \neq y}}
c_{n,x,y}f(x)(f(x)-f(y))
\sum_{i=1}^{N}b_{i}(x)(h_{i}(x)-h_{i}(y))}\\
&\leq
\left(\frac{1}{2}\sum_{\substack{x ,y \in V_{n} \\x \neq y}}
c_{n,x,y}
(f(x))^{2}(f(x)-f(y))^{2}
\right)^{1/2} \\
&\quad \cdot \left(\frac{1}{2}
\sum_{\substack{x ,y \in V_{n} \\x \neq y}}
c_{n,x,y}
\left(\sum_{i=1}^{N}b_{i}(x)(h_{i}(x)-h_{i}(y))
\right)^{2} \right)^{1/2} \\
&\le \norm{f}_{\infty}
\mathcal{E}^{n}(f)^{1/2} \\
&\quad \cdot \left(\frac{1}{2}
\sum_{i,j=1}^{N}\sum_{\substack{x ,y \in V_{n} \\x \neq y}}
c_{n,x,y} b_{i}(x)b_{j}(x)(h_{i}(x)-h_{i}(y))(h_{j}(x)-h_{j}(y))
\right)^{1/2}.
\end{align}
From Proposition~\ref{kihon} and the inequality $uv \le \delta u^{2}+(4\delta)^{-1}v^2$ for $u,v  \in \nonnegaR$,
\begin{align}
\norm{f}_{\infty}
\mathcal{E}^{n}(f)^{1/2}
&\le \left(
\mathrm{diam}_{R_{n}}V_{n}^{1/2} \mathcal{E}^{n}(f)^{1/2}
+\norm{f}_{L^{2}(V_{n},\mu_{n})}
\right)
\mathcal{E}^{n}(f)^{1/2} \\
&\le 
\left(\mathrm{diam}_{R}X^{1/2}+\delta\right) 
\mathcal{E}^{n}(f)
+(4\delta)^{-1}\norm{f}_{L^{2}(V_{n},\mu_{n})}^{2}.
\end{align}
On the other hand, by Lemma \ref{importantb}, we have
\begin{align}
&\left(\frac{1}{2}
\sum_{i,j=1}^{N}\sum_{\substack{x ,y \in V_{n} \\x \neq y}}
c_{n,x,y} b_{i}(x)b_{j}(x)(h_{i}(x)-h_{i}(y))(h_{j}(x)-h_{j}(y))
\right)^{1/2} \\
&\quad \cdot \left( \mathrm{diam}_{R}X^{1/2}+\delta \right)< s, \\
&\left(\frac{1}{2}
\sum_{i,j=1}^{N}\sum_{\substack{x ,y \in V_{n} \\x \neq y}}
c_{n,x,y} b_{i}(x)b_{j}(x)(h_{i}(x)-h_{i}(y))(h_{j}(x)-h_{j}(y))
\right)^{1/2} \\
& \quad \cdot (4\delta)^{-1}< t
\end{align}
for large enough $n \in \N$. This proves the assertion.
\end{proof}
In the same way, we can prove the following:
\begin{lem}
    We have
\begin{equation}
\abs{\mathcal{Q}(f)}
\leq 
s\mathcal{E}(f)+t\norm{f}^{2}_{L^{2}(X,\mu)}, 
\quad f \in \mathcal{F}.
\end{equation}
\end{lem}
\begin{prp}\label{prp5}
For large enough $n \in \N$, we have
\begin{equation}
(1-s)\mathcal{E}^{n}_{\lambda}(f)
\leq
\mathcal{A}^{n}_{\lambda}(f)
\leq
(1+s)\mathcal{E}^{n}_{\lambda}(f), \quad f\in l(V_{n}).
\end{equation}
\end{prp}
\begin{proof}
Since $\mathcal{A}^{n}=\mathcal{E}^{n}+\mathcal{Q}^{n}$ holds, 
for large enough $n \in \N$ and any $f\in l(V_{n})$, 
\begin{equation}
(1-s)\mathcal{E}^{n}(f)
-t\norm{f}^{2}_{L^{2}(V_n, \mu_{n})}
\leq
\mathcal{A}^{n}(f)
\leq
(1+s)\mathcal{E}^{n}(f)
+t\norm{f}^{2}_{L^{2}(V_n, \mu_{n})}
\end{equation}
holds by Lemma~\ref{daigilem}. Thus, we have
\begin{align}
(1-s)\mathcal{E}^{n}_{\lambda}(f)
&=(1-s)\mathcal{E}^{n}(f)
+(1-s)\lambda \norm{f}^{2}_{L^{2}(V_n, \mu_{n})} \\
&=\left\{
(1-s)\mathcal{E}^{n}(f)
-t\norm{f}^{2}_{L^{2}(V_n, \mu_{n})}
\right\}
+\lambda \norm{f}^{2}_{L^{2}(V_n, \mu_{n})} \\
&\leq \mathcal{A}^{n}(f)
+\lambda \norm{f}^{2}_{L^{2}(V_n, \mu_{n})}
(=\mathcal{A}^{n}_{\lambda}(f)) \\
&\leq
\left\{
(1+s)\mathcal{E}^{n}(f)
+t\norm{f}^{2}_{L^{2}(V_n, \mu_{n})}
\right\}
+\lambda \norm{f}^{2}_{L^{2}(V_n, \mu_{n})} \\
&=(1+s)\mathcal{E}^{n}(f)
+(1+s)\lambda \norm{f}^{2}_{L^{2}(V_n, \mu_{n})} \\
&=(1+s)\mathcal{E}^{n}_{\lambda}(f).
\qedhere \end{align}
\end{proof}
Similarly, the following holds:
\begin{prp}\label{prp6}
    We have
\begin{equation}
(1-s)\mathcal{E}_{\lambda}(f)
\leq
\mathcal{A}_{\lambda}(f)
\leq
(1+s)\mathcal{E}_{\lambda}(f), \quad f \in \mathcal{F}.
\end{equation}
\end{prp}
\begin{rem}
The above inequality still holds if $\lambda$ is replaced with $\alpha \in (\lambda, \infty)$.
\end{rem}
\begin{prp}\label{semidir1}
For large enough $n \in \N$, $(\mathcal{A}^{n},l(V_{n}))$ is a regular 
semi-Dirichlet form on $L^{2}(V_{n},\mu_{n})$ with the constant $\lambda$ in $\mathrm{(SD1)}$.
\end{prp}
\begin{proof}
If it is shown that $(\mathcal{A}^{n},l(V_{n}))$ is a semi-Dirichlet form on $L^{2}(V_{n},\mu_{n})$, then $(\mathcal{A}^{n},l(V_{n}))$ is regular since its domain is $l(V_{n})$. $\mathrm{(SD1)}$ follows from Proposition~\ref{prp5}.
$\mathrm{(SD2)}$ follows from the completeness of 
$\mathcal{E}^{n}_{\lambda+1}(\cdot)^{1/2}$ and Proposition~\ref{prp5}.

We prove $\mathrm{(SD3)}$. First, note that for any $x,y \geq 0$ and $p >q \geq 0$, 
the inequality $qx+y \le p\left\{x^{2}+(p^2-q^2)^{-1}y^{2} \right\} ^{1/2}$ 
holds. For any $g\in l(V_{n})$, by putting
$p=\left(\mathrm{diam}_{R}X\right)^{1/2}+2\delta
,q=\left(\mathrm{diam}_{R}X\right)^{1/2},
x=\mathcal{E}^{n}(g)^{1/2},
y=\norm{g}_{L^{2}(V_{n},\mu_{n})}$ in this inequality and using Proposition~\ref{kihon}, we have
\begin{align}
\norm{g}_{\infty} 
&\leq
\left( \mathrm{diam}_{R_{n}}V_{n} \right)^{1/2} \mathcal{E}^{n}(g)^{1/2}
+\norm{g}_{L^{2}(V_{n},\mu_{n})} \\
&\leq
\left( \mathrm{diam}_{R}X\right)^{1/2} \mathcal{E}^{n}(g)^{1/2}
+\norm{g}_{L^{2}(V_{n},\mu_{n})} \\
&\leq
\left\{
\mathcal{E}^{n}(g)
+(4\delta)^{-1}
\left(
\left(\mathrm{diam}_{R}X\right)^{1/2}+\delta
\right)^{-1}
\norm{g}^{2}_{L^{2}(V_{n},\mu_{n})}
\right\}^{1/2}\\
&\quad \cdot \left( \left(\mathrm{diam}_{R}X\right)^{1/2}+2\delta \right)\\
&=\mathcal{E}^{n}_{\lambda}(g)^{1/2}
\left( \left(\mathrm{diam}_{R}X\right)^{1/2}+2\delta \right).
\end{align}
On the other hand, for $1 \leq i  \leq N$, 
we have, in the same way as Lemma~\ref{daigilem}, 
\begin{align}
&\abs{\mathcal{Q}^{n}_{i}(f,g)}
\leq
\norm{b_{i}}_{\infty}
\mathcal{E}(h_{i})^{1/2}
\norm{g}_{\infty}
\mathcal{E}^{n}(f)^{1/2}, 
\quad f,g\in l(V_{n}),
\end{align}
where we use the inequality $\mathcal{E}^{n}(h_{i}\vert_{V_{n}}) \le \mathcal{E}(h_{i})$ (see Remark~\ref{entraceineq}). Thus, we get, for any $f,g\in l(V_{n})$,
\begin{align}
\abs{\mathcal{Q}^{n}(f,g)}
&\leq
\sum_{i=1}^{N}\abs{\mathcal{Q}^{n}_{i}(f,g)} \\
&\leq
\norm{g}_{\infty}
\mathcal{E}^{n}(f)^{1/2} 
\sum_{i=1}^{N}
\norm{b_{i}}_{\infty}
\mathcal{E}(h_{i})^{1/2} \\
&\leq
\left( \left(\mathrm{diam}_{R}X \right)^{1/2}+2\delta \right)
\mathcal{E}^{n}_{\lambda}(f)^{1/2}
\mathcal{E}^{n}_{\lambda}(g)^{1/2}
\sum_{i=1}^{N}
\norm{b_{i}}_{\infty}
\mathcal{E}(h_{i})^{1/2}.
\end{align}
Thus, together with Proposition~\ref{prp5}, we have 
\begin{align}
\abs{\mathcal{A}^{n}(f,g)}
&\leq
\abs{\mathcal{E}^{n}(f,g)}
+\abs{\mathcal{Q}^{n}(f,g)} \\
&\leq
\left\{
1+\left( \left(\mathrm{diam}_{R}X\right)^{1/2}+2\delta \right)
\sum_{i=1}^{N}
\norm{b_{i}}_{\infty}
\mathcal{E}(h_{i})^{1/2}
\right\} \\
&\quad \cdot
\mathcal{E}^{n}_{\lambda}(f)^{1/2}
\mathcal{E}^{n}_{\lambda}(g)^{1/2}  \\
&\leq
(1-s)^{-1}
\left\{
1+\left( \left(\mathrm{diam}_{R}X\right)^{1/2}+2\delta \right)
\sum_{i=1}^{N}
\norm{b_{i}}_{\infty}
\mathcal{E}(h_{i})^{1/2}
\right\} \\
&\quad \cdot
\mathcal{A}^{n}_{\lambda}(f)^{1/2}
\mathcal{A}^{n}_{\lambda}(g)^{1/2}, 
\quad f,g\in l(V_{n}).
\end{align}
Lastly, we prove $\mathrm{(SD4)}$.
Let $f\in l(V_{n})$ and $a \geq 0$. For distinct $x,y \in V_{n}$, we define $r(x,y)$ by 
\begin{align}
r(x,y) 
&\eqdot
\{(f \wedge a)(x)-(f \wedge a)(y)\}
\left\{
(f-f \wedge a)(x)-(f-f \wedge a)(y)
\vphantom{\sum_{i=1}^{N}
b_{i}(x)(f-f \wedge a)(x)
(h_{i}(x)-h_{i}(y))}\right. \\
&\quad{} \left.  
+\sum_{i=1}^{N} b_{i}(x)(h_{i}(x)-h_{i}(y))(f-f \wedge a)(x)\right\}.
\end{align}
Then, a direct computation gives
\begin{equation}\label{markovdiscrete}
\mathcal{A}^{n}(f \wedge a,f-f \wedge a)
=\frac{1}{2}
\sum_{\substack{x ,y \in V_{n}x \neq y}}
c_{n,x,y}r(x,y).
\end{equation}
Now, for large enough $n$, we prove each term of the right-hand side of \eqref{markovdiscrete} is non-negative. To do this, we need to consider the 
following four cases:
\begin{enumerate}[label=\normalfont(\roman*)]
    \item In the case of $f(x) \geq a \, \text{and} \, f(y) \geq a$,
    we have $r(x,y)=0$, thus we have $c_{n,x,y}r(x,y)=0$.
    \item In the case of $f(x) < a \, \text{and} \,  f(y) < a$,
    we have $r(x,y)=0$, thus we have $c_{n,x,y}r(x,y)=0$.
    \item In the case of $f(x) < a \, \text{and} \, f(y) \geq a$, 
    we have $r(x,y)=(f(x)-a)(a-f(y)) \geq 0$, thus 
    we have $c_{n,x,y}r(x,y) \ge 0$.
    \item In the case of $f(x) \geq a \, \text{and} \, f(y) < a$, we have
    \begin{align}
        r(x,y)=(a-f(y))(f(x)-a)
        \left\{
        1+
        \sum_{i=1}^{N}
        b_{i}(x)(h_{i}(x)-h_{i}(y))
        \right\}.
    \end{align}
First, we consider the case of Assumption A. By Condition \ref{assmp2}, we get 
\begin{align}
    \abs{\sum_{i=1}^{N}b_{i}(x)(h_{i}(x)-h_{i}(y))}
    &\le R(x,y)^{1/2}
    \mathcal{E}\left(\sum_{i=1}^{N}b_{i}(x)h_{i}\right)^{1/2}\\
    &\le1,
\end{align}
which implies $c_{n,x,y}r(x,y) \ge 0$.

Second, we consider the case of Assumption B. For any real function $h$ on $X$ and any subset $A$ of $X$, we define $\mathop{\mathrm{Osc}}(h, A)$ by 
\begin{align}
\mathop{\mathrm{Osc}}(h, A) \coloneqq \sup_{x \in A}h(x)-\inf_{x \in A}h(x).
\end{align}
 Since Lemma~\ref{cond3first}~\ref{enum:1condition3}
holds and $h_{i}$ is 
uniformly continuous on the compact space $(X,R)$, for large enough $n \in \N$, we have
\begin{align}\label{localosc}
    \sup_{1 \le i \le N} \left(\norm{b_{i}}_{\infty}
    \sup_{\lambda \in \Lambda_{n}}\mathop{\mathrm{Osc}}(h_{i}, \mathrm{Cl}(U_{\lambda}))\right) 
    \le N^{-1}.
\end{align}
If there is no $\lambda \in \Lambda_{n}$ such that $x, y \in \partial U_{\lambda}$, 
then, from Lemma~\ref{localcoef}, we have $c_{n,x,y}=0$, thus $c_{n,x,y}r(x,y)=0$ 
holds. 

If there is some $\lambda \in \Lambda_{n}$ such that $x, y \in \partial U_{\lambda}$, 
we have, from \eqref{localosc}, 
\begin{align}
1+\sum_{i=1}^{N}b_{i}(x)(h_{i}(x)-h_{i}(y))\geq 0.
\end{align}
Thus, we get $c_{n,x,y}r(x,y) \ge 0$.
\qedhere \end{enumerate}
\end{proof}
\begin{rem}\label{markovchainrepresentation}
Let $L_{n}$ be the generator of the semigroup on $L^2(V_{n},\mu_{n})$ associated with $(\mathcal{A}^{n},l(V_{n}))$. Then, for 
any $f \in l(V_{n})$ and $x \in V_{n}$, we have 
\begin{align}
-\mu_{n}(x)(L_{n}f)(x) &=(-L_{n}f, 1_{\{x\}})_{L^2(V_{n},\mu_{n})} \\
&=\mathcal{A}^{n}(f,1_{\{x\}}) \\
&=\sum_{y \in V_{n}\setminus \{x\}}
c_{n,x,y}(f(x)-f(y))
+\frac{1}{2}\sum_{i=1}^{N}\sum_{y \in V_{n}\setminus \{x\}}
c_{n,x,y}b_{i}(x) \\
&\quad \cdot (h_{i}(x)-h_{i}(y))(f(x)-f(y)).
\end{align}
Thus, we obtain 
\begin{align}
(L_{n}f)(x)=\mu_{n}(x)^{-1}
\sum_{y \in V_{n}\setminus \{x\}}
c_{n,x,y}(1+\eta(x,y))(f(y)-f(x)),
\end{align}
where $\eta(x,y) \eqdot \frac{1}{2}\sum_{i=1}^{N}b_{i}(x)(h_{i}(x)-h_{i}(y))$. We represent $L_{n}$ as a matrix $\{L_{n}(x,y)\}_{x,y \in V_{n}}$. Then, $L_{n}(x,y)=\mu_{n}(x)^{-1}
c_{n,x,y}(1+\eta(x,y))$ for distinct $x,y \in V_{n}$ and 
$L_{n}(x,x)=-\mu_{n}(x)^{-1}
\sum_{y \in V_{n}\setminus \{x\}}
c_{n,x,y}(1+\eta(x,y))$ for $x \in V_{n}$. 
For any $x \in V_{n}$, we define $q_{n}(x)$ by 
\begin{align}
q_{n}(x)
&=-L_{n}(x,x) \\
&=\mu_{n}(x)^{-1}
\sum_{y \in V_{n}\setminus \{x\}}
c_{n,x,y}(1+\eta(x,y)).
\end{align}
For any distinct $x,y \in V_{n}$, we define $\pi_{n}(x,y)$ by
\begin{align}
\pi_{n}(x,y)
&=q_{n}(x)^{-1}L_{n}(x,y) \\
&=\frac{c_{n,x,y}(1+\eta(x,y))}{\sum_{z \in V_{n}\setminus \{x\}}
c_{n,x,z}(1+\eta(x,z))}.
\end{align}
Let $Y_{n}$ be the Markov process on $V_{n}$ associated with $(\mathcal{A}^{n},l(V_{n}))$. Then, the holding time of $Y_{n}$ at $x \in V_{n}$ is exponentially distributed with parameter $q_{n}(x)$ and $Y_{n}$ jumps from $x$ to $y$ with probability $\pi_{n}(x,y)$.
\end{rem}
\begin{prp}\label{semidir2}
$(\mathcal{A},\mathcal{F})$ is a regular semi-Dirichlet form on $L^{2}(X,\mu)$ with the constant $\lambda$ in $\mathrm{(SD1)}$.
\end{prp}
\begin{proof}
$\mathrm{(SD1)}$ is clear from Proposition~\ref{prp6}.
$\mathrm{(SD2)}$ follows from the completeness of 
$\mathcal{E}_{\lambda+1}^{1/2}(\cdot)$ and Proposition~\ref{prp6}.
$\mathrm{(SD3)}$ is shown similarly as Proposition~\ref{semidir1}.

We prove $\mathrm{(SD4)}$. We take arbitrary $f \in \mathcal{F}$ and $a \geq 0$.
From the Markov property of  $(\mathcal{E},\mathcal{F})$, we get
\begin{equation}
f \wedge a \in \mathcal{F} \,\text{ and }\,
\mathcal{E}(f \wedge a,f-f \wedge a) \geq 0.
\end{equation}
On the other hand, for $1\leq i \leq N$, 
\begin{align}
    \mathcal{Q}_{i}(f \wedge a,f-f \wedge a)
    &=\frac{1}{2}\int_{\{f < a\}}b_{i}(f-f \wedge a)\,d\nu_{f \wedge a,h_{i}} \\
    &\quad +\frac{1}{2}\int_{\{f \ge a\}}b_{i}(f-f \wedge a)\,d\nu_{f \wedge a,h_{i}} \\
    &=\frac{1}{2}\int_{\{f \ge a\}}b_{i}(f-f \wedge a)\,d\nu_{f \wedge a,h_{i}}
\end{align}
holds. Now, because of the strong locality of $(\mathcal{E},\mathcal{F})$, 
$\nu_{f \wedge a}((f \wedge a)^{-1} (\{ a\}))=0$ holds (see \cite[Theorem 4.3.8]{chen2012symmetric}). Since $\{f \ge a\}=(f \wedge a)^{-1} (\{ a\})$ holds, we obtain $\nu_{f \wedge a}(\{f \ge a\})=0$. 
Thus, $\nu_{f \wedge a,h_{i}}(\{f \ge a\})=0$ holds. Therefore, we get 
\begin{equation}
    \mathcal{Q}_{i}(f \wedge a,f-f \wedge a)
    =\frac{1}{2}\int_{\{f \ge a\}}b_{i}(f-f \wedge a)\,d\nu_{f \wedge a,h_{i}}
    =0.
    \end{equation}
Thus, we get $\mathcal{A}(f \wedge a,f-f \wedge a)\geq0$, which proves 
$\mathrm{(SD4)}$. Lastly, from Remark~\ref{regularremark}, $(\mathcal{A},\mathcal{F})$ is regular.
\qedhere
\end{proof}
\begin{prp}\label{localstrong}
The semi-Dirichlet form $(\mathcal{A},\mathcal{F})$ is strongly local.
\end{prp}
\begin{proof}
Let $f,g \in \mathcal{F}$ and assume that there exists an open set $U \subset X$
such that $\mathop{\mathrm{supp}} g \subset U$ and $f \lvert_{U}$ is a constant function.
Then, $\mathcal{E}(f,g)=0$ holds because of the strong locality of $(\mathcal{E},\mathcal{F})$.

Next, we take an arbitrary $1 \leq i \leq N$. Then, since $\nu_{f}(U)=0$ holds from 
Lemma~\ref{enerzero}, we get $\nu_{f,h_{i}}=0$ on $U$. Therefore,
\begin{align}
\mathcal{Q}_{i}(f,g) 
&= \frac{1}{2}\int_{X} b_{i}g \,d\nu_{f,h_{i}} \\
&= \frac{1}{2}\int_{U} b_{i}g \,d\nu_{f,h_{i}} \quad (\text{from } \mathop{\mathrm{supp}} g \subset U)\\
&=0.
\end{align}
Thus, $\mathcal{A}(f,g)=0$ holds.
\end{proof}
\begin{prp}\label{capazero}
Let $\mathop{\mathrm{Cap}}\nolimits^{(\alpha)}_{\mathcal{A}}$ denote the $\alpha$-capacity with respect to $(\mathcal{A},\mathcal{F})$
 for $\alpha > \lambda$. Then, $\mathop{\mathrm{Cap}}\nolimits^{(\alpha)}_{\mathcal{A}}(N)=0$ implies $N = \emptyset$ for $N \subset X$.
\end{prp}
\begin{proof}
    Let $\mathop{\mathrm{Cap}}\nolimits^{(\alpha)}_{\mathcal{E}}$ denote the $\alpha$-capacity with respect to $(\mathcal{E},\mathcal{F})$. 
    First, we show that $(1-s)
    \mathop{\mathrm{Cap}}\nolimits^{(\alpha)}_{\mathcal{E}}
    \le
    \mathop{\mathrm{Cap}}\nolimits^{(\alpha)}_{\mathcal{\mathcal{A}}}
    $ holds. We take an arbitrary open set $U$. We may assume that $\mathop{\mathrm{Cap}}\nolimits^{(\alpha)}_{\mathcal{\mathcal{A}}}(U) < \infty$. Note that $\mathscr{L}_{U} \neq \emptyset$. 
    Let $e^{\alpha}_{\mathcal{A},U}$ (resp.\ $\hat{e}^{\alpha}_{\mathcal{A},U}$) be the $\alpha$-equilibrium potential (resp.\ $\alpha$-coequilibrium potential) of $U$ with respect to $(\mathcal{A},\mathcal{F})$. From Definition~\ref{capadef}, Remark~\ref{symcapa}, and Proposition~\ref{prp6}, we have 
    \begin{align}
        (1-s)\mathop{\mathrm{Cap}}\nolimits^{(\alpha)}_{\mathcal{E}}(U)
        &\le (1-s)\mathcal{E}_{\alpha}(e^{\alpha}_{\mathcal{A},U},e^{\alpha}_{\mathcal{A},U})\\
        &\le \mathcal{A}_{\alpha}(e^{\alpha}_{\mathcal{A},U},e^{\alpha}_{\mathcal{A},U})\\
        &\le \mathcal{A}_{\alpha}(e^{\alpha}_{\mathcal{A},U},\hat{e}^{\alpha}_{\mathcal{A},U})\\
        &=\mathop{\mathrm{Cap}}\nolimits^{(\alpha)}_{\mathcal{\mathcal{A}}}(U).
    \end{align}
    Therefore, we obtain $(1-s)\mathop{\mathrm{Cap}}\nolimits^{(\alpha)}_{\mathcal{E}}
    \le\mathop{\mathrm{Cap}}\nolimits^{(\alpha)}_{\mathcal{\mathcal{A}}}$. Suppose that $\mathop{\mathrm{Cap}}\nolimits^{(\alpha)}_{\mathcal{A}}(N)=0$ for $N \subset X$. Then, we have $\mathop{\mathrm{Cap}}\nolimits^{(\alpha)}_{\mathcal{E}}(N)=0$. 
    Since $(\mathcal{E},\mathcal{F})$ is a resistance form, we obtain $N = \emptyset$ (see \cite[Theorem 9.9]{MR2919892} for details).
\end{proof}
\section{Main theorem}
\subsection{Convergence of resolvents}
The following three lemmas are easily proven and their proofs are omitted.
\begin{lem}\label{techn2}
Let $\{ f^{(n)}\}_{n \in \N},\{ g^{(n)}\}_{n \in \N} \subset \mathcal{F}$ and $f,g \in \mathcal{F}$. Assume that $\{ f^{(n)}\}_{n \in \N}$ weakly converges to $f$ in the Hilbert space 
$(\mathcal{F}, \mathcal{E}_{1}(\cdot))$ and $\limtoinfty{n}\mathcal{E}(g^{(n)}-g)=0$. Then, 
$\limtoinfty{n}\mathcal{E}(f^{(n)},g^{(n)})=\mathcal{E}(f,g)$.
\end{lem}
\begin{lem} \label{techn3}
For any $1 \leq i \leq N$, $n \in \N$ and $f,g \in l(V_{n})$, we have
\begin{align}
    2\mathcal{Q}^{n}_{i}(g,f)
    &=
    \mathcal{E}^{n}(b_{i}\lvert_{V_{n}}fg,h_{i}\lvert_{V_{n}})
    +\mathcal{E}^{n}(b_{i}\lvert_{V_{n}}fh_{i}\lvert_{V_{n}},g)
    -\mathcal{E}^{n}(b_{i}\lvert_{V_{n}}f,gh_{i}\lvert_{V_{n}}).
\end{align}
\end{lem}
\begin{lem}\label{techn4}
    We take an arbitrary $1 \leq i \leq N$ and assume that $b_{i} \in \mathcal{F}$. Then, we have
    \begin{equation}
2\mathcal{Q}_{i}(g,f)
=\mathcal{E}(b_{i}fg,h_{i})
+\mathcal{E}(b_{i}fh_{i},g)
-\mathcal{E}(b_{i}f,gh_{i}), 
\quad f,g \in \mathcal{F}.
\end{equation}
\end{lem}
In the next lemma, when we need to express the dependence on $\bm{b} \coloneqq (b_1,\ldots,b_{N})$ clearly, we will write $Q(\bm{b})$, $Q_{i}(b_{i})$, $Q^{n}(\bm{b})$ and $Q^{n}_{i}(b_{i})$ instead of $Q$, $Q_{i}$, $Q^{n}$ and $Q^{n}_{i}$, respectively.
\begin{lem}\label{f2lemma}
Let $n_{k} \uparrow \infty, f_{k} \in L^2(V_{n_{k}},\mu_{n_{k}}), f,g  \in \mathcal{F}$ and assume that $\{ f_{k} \}_{k \in \N}$ weakly K-S converges to $f$ and 
$\supover{k} \mathcal{E}^{n_{k}}_{1}(f_{k}) < \infty$. Then, we have
\begin{equation}
\liminftoinfty{k}\mathcal{Q}^{n_{k}}(g\lvert_{V_{n_{k}}},f_{k})
\leq \mathcal{Q}(g,f).
\end{equation}
\end{lem}
\begin{proof}
The proof is given by the following two steps.
\begin{enumerate}[label=\normalfont(\roman*)]
\item\label{prooffirsto} If $b_{i} \in \mathcal{F}$ holds for all $1 \leq i \leq N$, we put 
$g^{\prime}_{i} \eqdot b_{i}g \in \mathcal{F}, h^{\prime}_{i} \eqdot  b_{i}h_{i} \in \mathcal{F}$. 
Then, by Lemma~\ref{techn3}, we have 
{\small
\begin{align}\label{calcul1}
    2\mathcal{Q}^{n_{k}}_{i}(g\lvert_{V_{n_{k}}},f_{k})
    &=
    \mathcal{E}^{n_{k}}
    (b_{i}\lvert_{V_{n_{k}}}f_{k}g\lvert_{V_{n_{k}}},h_{i}\lvert_{V_{n_{k}}})
    +\mathcal{E}^{n_{k}}
    (b_{i}\lvert_{V_{n_{k}}}f_{k}h_{i}\lvert_{V_{n_{k}}},g\lvert_{V_{n_{k}}}) \notag \\
    &\quad-\mathcal{E}^{n_{k}}
    (b_{i}\lvert_{V_{n_{k}}}f_{k},g\lvert_{V_{n_{k}}}h_{i}\lvert_{V_{n_{k}}}) \notag \\
    &=
    \mathcal{E}^{n_{k}}
    (g^{\prime}_{i}\lvert_{V_{n_{k}}}f_{k},h_{i}\lvert_{V_{n_{k}}})
    +\mathcal{E}^{n_{k}}
    (h^{\prime}_{i}\lvert_{V_{n_{k}}}f_{k},g\lvert_{V_{n_{k}}}) \notag \\
    &\quad{} -\mathcal{E}^{n_{k}}
    (b_{i}\lvert_{V_{n_{k}}}f_{k},(gh_{i})\lvert_{V_{n_{k}}}) \notag \\
    &=
    \mathcal{E}
    (\widehat{ g^{\prime}_{i}\lvert_{ V_{n_{k}}}f_{k}},
    \widehat{h_{i}\lvert_{V_{n_{k}}}})
    +\mathcal{E}
    (\widehat{h^{\prime}_{i}\lvert_{V_{n_{k}}}f_{k}},
    \widehat{g\lvert_{V_{n_{k}}}}) \notag \\
    &\quad{} -\mathcal{E}
    (\widehat{b_{i}\lvert_{V_{n_{k}}}f_{k}},
    \widehat{(gh_{i})\lvert_{V_{n_{k}}}}).
    \end{align}
}
Since $\{ f_{k} \}_{k \in \N}$ weakly K-S converges to $f$ and 
$\supover{k} \mathcal{E}^{n_{k}}_{1}(f_{k}) < \infty$ holds, 
$\{ \widehat{g^{\prime}_{i}\lvert_{ V_{n_{k}}}f_{k}} \}_{k \in \N}$
$($resp.\ $\{ \widehat{h^{\prime}_{i}\lvert_{V_{n_{k}}}f_{k}} \}_{k \in \N}$, 
$\{ \widehat{b_{i}\lvert_{V_{n_{k}}}f_{k}} \}_{k \in \N}$$)$
weakly converges to $g^{\prime}_{i}f$
$($resp.\ $h^{\prime}_{i}f$, $b_{i}f$$)$ in $(\mathcal{F}, \mathcal{E}_{1}(\cdot))$ 
by Proposition~\ref{seki}. On the other hand, by Lemma~\ref{techn1}, 
\begin{align}
&\limtoinfty{k}\mathcal{E}(\widehat{h_{i}\lvert_{V_{n_{k}}}}-h_{i})=0, \\
&\limtoinfty{k}\mathcal{E}(\widehat{g\lvert_{V_{n_{k}}}}-g)=0, \\
&\limtoinfty{k}\mathcal{E}(\widehat{gh_{i}\lvert_{V_{n_{k}}}}-gh_{i})=0 
\end{align}
hold. Thus, by Lemma~\ref{techn2},
\begin{align}
    &\limtoinfty{k}
    \mathcal{E}
    (\widehat{ g^{\prime}_{i}\lvert_{ V_{n_{k}}}f_{k}},
    \widehat{h_{i}\lvert_{V_{n_{k}}}})
    =\mathcal{E}(g^{\prime}_{i}f,h_{i})
    =\mathcal{E}(b_{i}gf,h_{i}),\\
    &\limtoinfty{k}
    \mathcal{E}
    (\widehat{h^{\prime}_{i}\lvert_{V_{n_{k}}}f_{k}},
    \widehat{g\lvert_{V_{n_{k}}}})
     =\mathcal{E}(h^{\prime}_{i}f,g)
     =\mathcal{E}(b_{i}h_{i}f,g),\\
    &\limtoinfty{k}
    \mathcal{E}
    (\widehat{b_{i}\lvert_{V_{n_{k}}}f_{k}},
    \widehat{(gh_{i})\lvert_{V_{n_{k}}}})
    =\mathcal{E}(b_{i}f,gh_{i})
\end{align}
hold. Thus, together with the above equations with \eqref{calcul1}, 
\begin{align}
\limtoinfty{k}
\mathcal{Q}^{n_{k}}_{i}(g\lvert_{V_{n_{k}}},f_{k})
&=\frac{1}{2}
\left(
\mathcal{E}(b_{i}gf,h_{i})+\mathcal{E}(b_{i}h_{i}f,g)-\mathcal{E}(b_{i}f,gh_{i}) 
\right)\\
&=\mathcal{Q}_{i}(g,f)
\end{align}
holds by Lemma~\ref{techn4}. Therefore, we have 
\begin{equation}
\limtoinfty{k}
\mathcal{Q}^{n_{k}}(g\lvert_{V_{n_{k}}},f_{k}) 
=\mathcal{Q}(g,f).
\end{equation}
\item In the general case, we take an arbitrary $\varepsilon > 0$. 
Since $\mathcal{F}$ is dense in $(C(X), \norm{\cdot}_{\infty})$, there exists 
$\bm{b}^{\prime} = (b^{\prime}_{1}, \ldots ,b^{\prime}_{N}) \in \mathcal{F}^{N}$ such that $\norm{b_{i}-b^{\prime}_{i}}_{\infty} < \varepsilon$ for $i=1,\ldots,N$.
For $k \in \N$ and $1 \le i \le N$, in the same way as Lemma~\ref{daigilem}, we have 
\begin{align}
&\abs{\mathcal{Q}^{n_{k}}_{i}(b_{i})(g\lvert_{V_{n_{k}}},f_{k})
-\mathcal{Q}^{n_{k}}_{i}(b^{\prime}_{i})(g\lvert_{V_{n_{k}}},f_{k})}\\
&=\abs{\mathcal{Q}^{n_{k}}_{i}(b_{i}-b^{\prime}_{i})(g\lvert_{V_{n_{k}}},f_{k})}\\
&\le \norm{b_{i}-b^{\prime}_{i}}_{\infty}\norm{f_{k}}_{\infty}
\mathcal{E}^{n_{k}}(g\lvert_{V_{n_{k}}})^{1/2}
\mathcal{E}^{n_{k}}(h_{i} \lvert_{V_{n_{k}}})^{1/2} \\
&\leq
\varepsilon
\mathcal{E}(g)^{1/2}
\mathcal{E}(h_{i})^{1/2}
\norm{f_{k}}_{\infty} \\
&\leq
\varepsilon
\mathcal{E}(g)^{1/2}
\mathcal{E}(h_{i})^{1/2}
\left( \left(\mathrm{diam}_{R}X\right)^{1/2}+1 \right)
\supover{m}
 \mathcal{E}^{n_{m}}_{1}(f_{m})^{1/2}
\end{align}
and 
\begin{align}
    \abs{\mathcal{Q}_{i}(b^{\prime}_{i})(g,f)-\mathcal{Q}_{i}(b_{i})(g,f)}
    &=\abs{
    \frac{1}{2}
    \int_{X}(b^{\prime}_{i}-b_{i})f\,d\nu_{g, h_{i}}}\\
    &\leq\frac{1}{2}
    \norm{b^{\prime}_{i}-b_{i}}_{\infty}
    \norm{f}_{\infty}
    \norm{\nu_{g, h_{i}}} \\
    &\leq\frac{1}{2}
    \varepsilon
    \norm{f}_{\infty}
    \mathcal{E}(g)^{1/2}
    \mathcal{E}(h_{i})^{1/2},   
\end{align}
where $\norm{\nu_{g, h_{i}}}$ denotes the total variation of $\nu_{g, h_{i}}$. 
We put $C_1,C_2 < \infty$ as
\begin{align}
&C_{1} \coloneqq \max_{1\le i \le N} \left\{ \mathcal{E}(g)^{1/2}
\mathcal{E}(h_{i})^{1/2}\left( \left(\mathrm{diam}_{R}X\right)^{1/2}+1 \right)
\supover{m}\mathcal{E}^{n_{m}}_{1}(f_{m})^{1/2} \right\},\\
&C_2 \coloneqq \frac{1}{2}\max_{1\le i \le N}\left( \norm{f}_{\infty}
    \mathcal{E}(g)^{1/2}\mathcal{E}(h_{i})^{1/2} \right).
\end{align}
Then, by \ref{prooffirsto} and the above inequalities, we have
\begin{align}
\liminftoinfty{k}\mathcal{Q}^{n_{k}}(\bm{b})(g\lvert_{V_{n_{k}}},f_{k})
&\le \liminftoinfty{k}\mathcal{Q}^{n_{k}}(\bm{b}^{\prime})(g\lvert_{V_{n_{k}}},f_{k})
+\varepsilon C_{1}N \\
&= \mathcal{Q}(\bm{b}^{\prime})(g,f)+\varepsilon C_{1}N \\
&\le \mathcal{Q}(\bm{b})(g,f)+\varepsilon C_{2}N+\varepsilon C_{1}N.
\end{align}
Thus, we obtain the conclusion by letting $\varepsilon \downarrow 0$.
\qedhere \end{enumerate}
\end{proof}
\begin{prp}\label{main1}
Let $\{G_{\alpha} \}_{\alpha >\lambda}$ $($resp.\ $\{G^{(n)}_{\alpha} \}_{\alpha >\lambda}$$)$ be 
the strongly continuous resolvents on $L^{2}(X,\mu)$ $($resp.\ $L^2(V_{n},\mu_{n})$$)$ associated with $(\mathcal{A},\mathcal{F})$ 
$($resp.\ $(\mathcal{A}^{n},l(V_{n}))$$)$. Then, for all $\alpha >\lambda$, 
$\{G^{(n)}_{\alpha} \}_{n \in \N}$ strongly K-S converges to $G_{\alpha}$.
\end{prp}
\begin{proof}
Since $L^{2}(X,\mu)$ is separable, it suffices to show $\mathrm{(F1)}$ and $\mathrm{(F2)}$ 
of Proposition~\ref{rslvconv}.
\begin{enumerate}
\item[$\mathrm{(F1)}$]
Let $n_{k} \uparrow \infty, f_{k} \in L^2(V_{n_{k}},\mu_{n_{k}}), f \in L^{2}(X,\mu)$ and 
assume that $\{ f_{k} \}_{k \in \N}$ weakly K-S converges to $f$ and 
$\supover{k} \mathcal{A}^{n_{k}}_{\lambda+1}(f_{k}) < \infty$ holds. Then, from Proposition~\ref{prp5}, we have
\begin{equation}
\supover{k} \mathcal{E}^{n_{k}}_{1}(f_{k})
\leq
\supover{k} \mathcal{E}^{n_{k}}_{\lambda+1}(f_{k})
\leq
(1-s)^{-1}\supover{k} \mathcal{A}^{n_{k}}_{\lambda+1}(f_{k}) < \infty.
\end{equation}
Therefore, we have $f \in \mathcal{F}$ from Corollary~\ref{cor2}.
\item[$\mathrm{(F2)}$]
Let $n_{k} \uparrow \infty, f_{k} \in L^2(V_{n_{k}},\mu_{n_{k}}), f,g  \in \mathcal{F}$ and 
assume that $\{ f_{k} \}_{k \in \N}$ weakly K-S converges to $f$ and 
$\supover{k} \mathcal{A}^{n_{k}}_{\lambda+1}(f_{k}) < \infty$ holds. Then, 
since $\{ g\lvert_{V_{n_{k}}} \}_{k \in \N}$ strongly K-S converges to $g$, it suffices to show 
\begin{equation}
\liminftoinfty{k}\mathcal{A}^{n_{k}}(g\lvert_{V_{n_{k}}},f_{k})
\leq \mathcal{A}(g,f).
\end{equation}
Since we have $\supover{k} \mathcal{E}^{n_{k}}_{1}(f_{k}) < \infty$ in the same way as $\mathrm{(F1)}$, 
$\{ \widehat{f_{k}} \}_{k \in \N}$ weakly converges to $f$ in the Hilbert space 
$(\mathcal{F}, \mathcal{E}_{1}(\cdot))$ by Corollary~\ref{cor2}. On the other hand, 
$\limtoinfty{k}\mathcal{E}(g-\widehat{g \lvert_{V_{n_{k}}}})=0$ holds by Lemma~\ref{techn1}. 
Thus, by Lemma~\ref{techn2}, we have 
\begin{equation}
\limtoinfty{k}
\mathcal{E}^{n_{k}}(g\lvert_{V_{n_{k}}},f_{k})
=
\limtoinfty{k}
\mathcal{E}(\widehat{g\lvert_{V_{n_{k}}}},\widehat{f_{k}})
=\mathcal{E}(g,f).
\end{equation}
On the other hand, since $\supover{k} \mathcal{E}^{n_{k}}_{1}(f_{k}) < \infty$, we have 
\begin{equation}
\liminftoinfty{k}\mathcal{Q}^{n_{k}}(g\lvert_{V_{n_{k}}},f_{k})
\leq \mathcal{Q}(g,f)
\end{equation}
from Lemma~\ref{f2lemma}. Therefore, we obtain the conclusion.
\qedhere \end{enumerate}
\end{proof}
\subsection{Convergence of Feller processes}
Let $\{ T_{t}\}_{t \geq 0}$ (resp.\ $\{ T^{n}_{t}\}_{t \geq 0}$), 
$\{G_{\alpha} \}_{\alpha > \lambda}$ (resp.\ $\{G^{n}_{\alpha} \}_{\alpha > \lambda}$) and $(L,\mathcal{D}(L))$ 
be the strongly continuous semigroup, resolvents and the generator on $L^2(X,\mu)$ associated with the regular semi-Dirichlet form $(\mathcal{A},\mathcal{F})$ (resp.\ $(\mathcal{A}^{n},l(V_{n}))$). 
Let $\mathbb{M}=((Z_{t})_{t \ge0},(P_{x})_{x \in X})$ (resp.\ $\mathbb{M}^{n}=((Y^{n}_{t})_{t \ge0},(P^{n}_{x})_{x \in V_{n}})$) be a Hunt process on $X$ (resp.\ $V_{n}$) associated with the regular semi-Dirichlet form $(\mathcal{A},\mathcal{F})$ (resp.\ $(\mathcal{A}^{n},l(V_{n}))$). Let $\{ p_t \}_{t\ge 0}$ (resp.\ $\{ p^{n}_t \}_{t\ge 0}$) and $\{R_{\alpha} \}_{\alpha >0}$ (resp.\ $\{R^{n}_{\alpha} \}_{\alpha >0}$) be the transition function and the resolvent of $\mathbb{M}$ (resp.\ $\mathbb{M}^{n}$).
In the following, for any function $f$ on $(X,\mu)$, we distinguish $f$ and its equivalence class $[f]$ consisting of all the functions 
which equal $f$ $\mu \text{-a.e.}$ 
\begin{rem}
In the following, also the strong Feller property of $\mathbb{M}$ is proven.
\end{rem}
\begin{proof}[Proof of Theorem~\ref{mainfeller}]
Note that $B_b(X)\subset L^2(X,\mu)$ holds since $\mu$ is a probability measure. Then, for any $t\ge0$ and $f \in B_b(X)$, $p_tf$ is quasi-continuous and $T_t[f]=[p_tf]$ holds. On the other hand, since the complexification of $\{ T_{t}\}_{t \geq 0}$ is the restriction of an analytic semigroup (see \cite[Corollary 2.21]{ma2012introduction}), we have 
$T_{t}(L^2(X,\mu)) \subset \mathcal{D}(L)$ (see \cite[Lemma~4.2]{pazy2012semigroups}) and especially $T_{t}(B_b(X)) \subset \mathcal{F}$. 
Thus, for any $f \in B_b(X)$, $T_t[f]$ has a continuous $\mu$-version $q_tf$. Since 
$p_tf=q_tf, \,\mu \text{-a.e.}$ and both $p_tf$ and $q_tf$ are quasi-continuous, we have $p_tf=q_tf, \, \text{q.e.}$ (cf.\ \cite[Lemma~2.2.1]{oshima2013semi}). Therefore, from Proposition~\ref{capazero}, $p_tf=q_tf$ holds and especially $p_tf$ is continuous. Thus, we have $p_t(B_b(X)) \subset C(X)$ and especially $p_t(C(X)) \subset C(X)$. 

Next, we prove $\lim_{t \downarrow 0}p_tf(x)=f(x)$ for any $f \in C(X), \, x \in X$. Since 
$\mathbb{M}$ is a Hunt process and $f$ is continuous, $f(Z_t)$ is right-continuous with respect to $t$. Thus, by Lebesgue's dominant convergence theorem, we obtain 
\begin{align}
    \lim_{t \downarrow 0}p_tf(x)&=\lim_{t \downarrow 0}E_{x}(f(Z_t))\\
    &=E_{x}(f(Z_0)) \\
    &=f(x).
\end{align}
Therefore, $\{ p_t \}_{t\ge 0}$ is Feller (see \cite[Theorem~17.6]{kallenberg}). 

Finally, we prove $p_t1=1$ for any $t\ge 0$. We have $\nu_{[1]}=0$ from $\nu_{[1]}(X)=2\mathcal{E}([1])=0$. Thus, by simple calculations, 
we have 
\begin{align}
\mathcal{A}_{\alpha}(\alpha^{-1}[1],f)
=\mathcal{A}(\alpha^{-1}[1],f)
+\alpha(\alpha^{-1}[1],f)_{L^2(X,\mu)}
=([1],f)_{L^2(X,\mu)}
\end{align}
for any $\alpha > \lambda$ and $f \in \mathcal{F}$. Thus, we obtain $G_{\alpha}[1]=\alpha^{-1}[1]$ 
for any $\alpha > \lambda$. In particular, $[1] \in G_{\alpha}(L^2(X,\mu))=\mathcal{D}(L)$ holds. 
Therefore, the following holds:
\begin{align}
   T_{t}[1]
   &=\limtoinfty{\alpha}e^{(\lambda-\alpha)t}
   \sum_{n=0}^{\infty}
   \frac{(t\alpha)^{n}}{n!}(\alpha G_{\alpha+\lambda})^{n}[1]
   =\limtoinfty{\alpha}e^{(\lambda-\alpha+\frac{\alpha^2}{\alpha+\lambda})t}
   [1]
   =[1]
\end{align}
(see \cite[Theorem 1.12]{ma2012introduction}). Thus, we obtain 
$p_{t}1=1, \,\mu \text{-a.e.}$ Since both $p_{t}1$ and $1$ are continuous and $\mu$ is full support, we have $p_{t}1=1$.
\end{proof}
\begin{rem}
    Since $\{ p_t \}_{t\ge 0}$ is Feller, $R_{\alpha}(C(X)) \subset C(X)$ holds for $\alpha >0$.
\end{rem}
Let $\{\tilde{R}_{\alpha} \}_{\alpha >0}$ be the strongly continuous resolvent 
on $(C(X),\norm{\cdot}_{\infty})$
associated with the strongly continuous contraction semigroup $\{ p_t\}_{t \geq 0}$ on $(C(X),\norm{\cdot}_{\infty})$.
\begin{prp}\label{resolventcoincide}
$\tilde{R}_{\alpha}f=R_{\alpha}f$ holds for any $f \in C(X)$.
\end{prp}
\begin{proof}
First note that, by using the Bochner integral on $(C(X),\norm{\cdot}_{\infty})$, $\tilde{R}_{\alpha}f$ is expressed as
\begin{align}
    \tilde{R}_{\alpha}f=
\int_{0}^{\infty}e^{-\alpha t}p_tf\,dt.
\end{align}
We take an arbitrary $x \in X$. We define $\iota_{x} \colon C(X) \rightarrow \R$ by 
$\iota_{x}(f)=f(x)$.
Since $\iota_{x}$ is a bounded operator from $(C(X),\norm{\cdot}_{\infty})$ 
to $\R$, by using Hille's theorem of the Bochner integral, we have
\begin{align}
(\tilde{R}_{\alpha}f)(x)
&=\iota_{x}(\tilde{R}_{\alpha}f) \\
&=\iota_{x}
\left(
\int_{0}^{\infty}e^{-\alpha t}p_tf\,dt
\right) \\
&=
\int_{0}^{\infty}
\iota_{x}\left(e^{-\alpha t}p_tf\right)\,dt \\
&=\int_{0}^{\infty}
e^{-\alpha t}(p_tf)(x) \,dt \\
&=(R_{\alpha}f)(x).
\qedhere \end{align}
\end{proof}
Next, for each $n \in \N$, we define the strongly continuous 
resolvent $\{\tilde{R}^{n}_{\alpha} \}_{\alpha >0}$ on $(l(V_{n}),\norm{\cdot}_{\infty})$ associated with the strongly continuous contraction semigroup $\{ p^{n}_t\}_{t \geq 0}$ on $(l(V_{n}),\norm{\cdot}_{\infty})$. The proposition stated above also holds in this case.
\begin{lem}\label{syuteiri}
For any $\alpha > \lambda$ and $f \in C(X)$, we have the following: 
\begin{equation}
\limtoinfty{n}\norm{
R^{n}_{\alpha}f\lvert_{V_{n}}
-(R_{\alpha}f)\lvert_{V_{n}}
}_{\infty}=0.
\end{equation}
\end{lem}
\begin{proof}
For the sake of simplicity, let $\mathscr{H}_{n}$ (resp.\ $\mathscr{H}$) denote $L^2(V_{n},\mu_{n})$ (resp.\ $L^2(X,\mu)$). 
For any $n \in \N$, by Proposition~\ref{prp5}, we have 
\begin{align}
\mathcal{E}^{n}_{1}({G}^{n}_{\alpha}f\lvert_{V_{n}})
\leq
\mathcal{E}^{n}_{\lambda+1}({G}^{n}_{\alpha}f\lvert_{V_{n}})
\leq
(1-s)^{-1}\mathcal{A}^{n}_{\lambda+1}({G}^{n}_{\alpha}f\lvert_{V_{n}}).
\end{align}
Thus we obtain $\supover{n}\mathcal{E}^{n}_{1}({G}^{n}_{\alpha}f\lvert_{V_{n}}) < \infty$
since the following holds:
\begin{align}
\mathcal{A}^{n}_{\lambda+1}({G}^{n}_{\alpha}f\lvert_{V_{n}})
&=(f\lvert_{V_{n}},{G}^{n}_{\alpha}f\lvert_{V_{n}})_{\mathscr{H}_{n}}
+(\lambda+1-\alpha)\norm{{G}^{n}_{\alpha}f\lvert_{V_{n}}}^{2}_{\mathscr{H}_{n}} \\
&\leq
\norm{{G}^{n}_{\alpha}}_{\mathrm{op}}
\norm{f\lvert_{V_{n}}}^{2}_{\mathscr{H}_{n}}
+\abs{\lambda+1-\alpha}
\norm{{G}^{n}_{\alpha}}^{2}_{\mathrm{op}}
\norm{f\lvert_{V_{n}}}^{2}_{\mathscr{H}_{n}} \\
&\leq
\left\{(\alpha-\lambda)^{-1}+\abs{\lambda+1-\alpha}(\alpha-\lambda)^{-2}\right\}
\norm{f}^{2}_{\infty} < \infty.
\end{align}
On the other hand, since $f\lvert_{V_{n}} \in \mathscr{H}_{n}$ strongly K-S converges 
to $[f] \in \mathscr{H}$ and ${G}^{n}_{\alpha}$ strongly K-S converges to ${G}_{\alpha}$ 
by Proposition~\ref{main1}, ${G}^{n}_{\alpha}f\lvert_{V_{n}} \in \mathscr{H}_{n}$ 
strongly K-S converges to ${G}_{\alpha}[f] \in \mathscr{H}$. Therefore, we obtain the conclusion 
from Corollary~\ref{cor2} and Proposition~\ref{cptop}, noticing that 
$R_{\alpha}f$ is continuous and $[R_{\alpha}f]=G_{\alpha}[f]$ 
(resp.\ $R^{n}_{\alpha}f\lvert_{V_{n}}=G_{\alpha}^{n}f\lvert_{V_{n}}$) holds.
\end{proof}
\begin{prp}\label{main2}
For any $t \geq 0$ and $f \in C(X)$, we have
\begin{equation}
\limtoinfty{n}\norm{
p^{n}_tf\lvert_{V_{n}}
-(p_tf)\lvert_{V_{n}}
}_{\infty}=0.
\end{equation}
\end{prp}
\begin{proof}
In this proof, we consider a different Kuwae--Shioya's framework from the above. 
Let $\Phi_{n} \colon C(X) \rightarrow l(V_{n})$ be defined as $\Phi_{n}(f)=f\lvert_{V_{n}}$. 
Then, the sequence of the Banach spaces 
$\{(l(V_{n}),\norm{\cdot}_{\infty})\}_{n \in \N}$ K-S converges to 
the Banach space $(C(X),\norm{\cdot}_{\infty})$ because $V_{*}$ is dense in $X$. Also, note 
that $\norm{\Phi_{n}}_{\mathrm{op}} \leq 1$ holds. 

We work in this framework. We take an arbitrary $\alpha > \lambda$. Then, for any 
$f \in C(X)$, by Lemma~\ref{resolventcoincide} and \ref{syuteiri}, we have 
\begin{equation}
\limtoinfty{n}\norm{
\tilde{R}^{n}_{\alpha}f\lvert_{V_{n}}
-(\tilde{R}_{\alpha}f)\lvert_{V_{n}}
}_{\infty}=0.
\end{equation}
Thus, $\tilde{R}^{n}_{\alpha}f\lvert_{V_{n}} \in (l(V_{n}),\norm{\cdot}_{\infty})$ 
strongly K-S converges to $\tilde{R}_{\alpha}f \in (C(X),\norm{\cdot}_{\infty})$
by Corollary~\ref{kyocor}. Moreover, since 
$\supover{n}\norm{\tilde{R}^{n}_{\alpha}}_{\mathrm{op}} \leq \alpha^{-1} < \infty$, 
$\{ \tilde{R}^{n}_{\alpha} \}_{n \in \N}$ strongly K-S converges to
$\tilde{R}_{\alpha}$ by Lemma~\ref{tech3}. Thus, by the fact that
$\supover{n}\norm{\Phi_{n}}_{\mathrm{op}} \leq 1 < \infty$ holds, 
$\{p^{n}_{t} \}_{n \in \N}$ strongly K-S converges to
$p_{t}$ from Proposition~\ref{daizi}. Thus, 
$\{p^{n}_{t}f\lvert_{V_{n}}\}_{n \in \N}$ 
strongly K-S converges to $p_{t}f$. From Corollary~\ref{kyocor}, 
we get the conclusion.
\end{proof}
\begin{proof}[Proof of Theorem~\ref{mainfirst}]
From Propositions~\ref{mainfeller} and \ref{main2}, $\{p_{t} \}_{t \geq 0 }$ is a conservative Feller semigroup 
on $(C(X),\norm{\cdot}_{\infty})$ and, 
for any $t \geq 0$ and $f \in C(X)$,
\begin{equation}
\limtoinfty{n}
\norm{p^{n}_{t}f \lvert_{V_{n}}-(p_{t}f)\lvert_{V_{n}}}_{\infty}
=0.
\end{equation}
Since the convergence of the initial distributions is assumed, the assertion follows from \cite[Section 4, Theorem 2.11]{ethier1986markov}.
\end{proof}
\section{Examples}
In this section, we review p.c.f.\ self-similar fractals as typical examples to which we 
can apply Theorem~\ref{mainfirst}.
\begin{dfn}[{\cite[Definition 1.3.1]{kigami2001analysis}}]
Let $S=\{ 1,2, \ldots,M \} \,(M \in \N, M \geq2)$ 
with discrete topology. 
We define $W \coloneqq S^{\N}$ with the product topology, 
$W_{n} \coloneqq S^{n}$ for $n \in \N$, 
$W_{*} \coloneqq \bigcup_{n=1}^{\infty}W_{n}$, 
$\sigma_{i} \colon W \ni w_{1}w_{2} \cdots \mapsto i w_{1}w_{2} \cdots \in W$ 
for $i \in S$, 
and $\sigma \colon W \ni w_{1}w_{2}w_{3} \cdots \mapsto w_{2}w_{3} \cdots \in W$.
\end{dfn}
\begin{dfn}[{\cite[Definitions 1.3.1, 1.3.4, 1.3.13]{kigami2001analysis}}]
Let $X$ be a compact metrizable connected topological space and 
$F_{i} \colon X \rightarrow X$ an injective continuous map for $i \in S$. Then, 
$(X, \{ F_{i} \}_{i \in S} )$ is called a self-similar structure if 
there exists a $($unique$)$ surjective continuous map $\pi \colon W \rightarrow X$ 
such that $F_{i} \circ \pi = \pi \circ \sigma_{i}$ holds for $i \in S$.
\end{dfn}
For a self-similar structure $(X, \{ F_{i} \}_{i \in S} )$, we define 
\begin{align}
        &C \coloneqq \pi^{-1} 
        \left(\bigcup_{\substack{i,j \in S \\i \neq j}}
        \left(F_{i}(X) \cap F_{j}(X) \right) \right)
        \subset W, \\
    &P \coloneqq 
    \bigcup_{n \geq 1} \sigma^{n}(C) \subset W.
    \end{align}
Then, $(X, \{ F_{i} \}_{i \in S} )$ is called post-critically finite
$($p.c.f.\ for short$)$ if $P$ is a finite set. 
For a p.c.f.\ self-similar structure $(X, \{ F_{i} \}_{i \in S} )$, 
we define $F_{w} \coloneqq F_{w_{1}} \circ F_{w_{2}} \circ \cdots \circ F_{w_{n}}$ 
for $w=w_{1}w_{2} \cdots w_{n} \in W_{n},\, n \in \N$ and 
\begin{align}
    &V_{0} \coloneqq \pi(P) \subset X ,\\
    &V_{n} \coloneqq 
    \bigcup_{w \in W_{n}}
    F_{w}(V_{0})
    \subset X
    \,(n \in \N).
\end{align}
Then, $\{V_{n} \}_{n \geq 0}$ is a non-decreasing sequence of nonempty finite 
subsets of $X$ and $V_{*} \coloneqq \bigcup_{n=0}^{\infty}V_{n}$ is dense in $X$.

Now, let $(X, \{ F_{i} \}_{i \in S} )$ be a p.c.f.\ self-similar structure and 
$\{r_{i} \}_{i \in S} \subset (0,1)$. We define 
$r_{w} \coloneqq r_{w_{1}} r_{w_{2}} \cdots r_{w_{n}}$ for 
$w=w_{1}w_{2} \cdots w_{n} \in W_{n}, n \in \N$. Let 
$(\mathcal{E}^{0},l(V_{0}))$ be a resistance form on $V_{0}$ represented as
\begin{equation}
\mathcal{E}^{0}(f_{0},g_{0})= 
\frac{1}{2}\sum_{\substack{x ,y \in V_{0} \\x \neq y}}
c_{0,x,y}(f_{0}(x)-f_{0}(y))(g_{0}(x)-g_{0}(y)),
\quad
f_{0},g_{0} \in l(V_{0}).
\end{equation}
For $n \in \N$, we define a resistance form $(\mathcal{E}^{n},l(V_{n}))$ 
on $V_{n}$ by 
\begin{align}
    \mathcal{E}^{n}(f_{n},g_{n}) \coloneqq
    \sum_{w \in W_{n}}
    r_{w}^{-1}
    \mathcal{E}^{0}(f_{n} \circ F_{w},g_{n}\circ F_{w}), 
    \quad f_{n},g_{n} \in l(V_{n}).
\end{align}
Note that $(\mathcal{E}^{n},l(V_{n}))$ is also represented as 
\begin{equation}
\mathcal{E}^{n}(f_{n},g_{n})= 
\frac{1}{2}\sum_{\substack{x ,y \in V_{n} \\x \neq y}}
c_{n,x,y}(f_{n}(x)-f_{n}(y))(g_{n}(x)-g_{n}(y)),
\quad
f_{n},g_{n} \in l(V_{n}),
\end{equation}
where, for $x,y \in V_{n} \,(x \neq y)$,  
\begin{align}
    c_{n,x,y}=\sum_{w \in W_{n}}1_{F_{w}(V_{0})}(x)1_{F_{w}(V_{0})}(y)
    r_{w}^{-1}c_{0,F_{w}^{-1}(x), F_{w}^{-1}(y)}.
\end{align}
Here, we assume $\mathrm{Tr}(\mathcal{E}^{1}\lvert V_{0})=\mathcal{E}^{0}$ holds. 
Then, $\{ \mathcal{E}^{n}(f \lvert_{V_{n}}) \}_{n \geq 0}$ is non-decreasing 
for any $f \in l(V_{*})$, thus we define 
\begin{align}
    &\mathcal{F} \coloneqq
    \{ f \in C(X) \mid
    \supover{n} \mathcal{E}^{n}(f \lvert_{V_{n}}) < \infty
    \}
\end{align}
and 
\begin{align}
    \mathcal{E}(f,g) \coloneqq
    \limtoinfty{n}
    \mathcal{E}^{n}(f \lvert_{V_{n}},g \lvert_{V_{n}}),
    \quad f,g \in \mathcal{F}.
\end{align}
Then, $(\mathcal{E},\mathcal{F})$ is a resistance form on $X$ and, 
for $n \in \N$, 
$\mathop{\mathrm{Tr}}(\mathcal{E} \lvert V_{n})=\mathcal{E}^{n}$ holds. 
Moreover, the resistance metric $R$ associated with $(\mathcal{E},\mathcal{F})$ 
defines the same topology as the original topology of $X$, and especially, 
$(X,R)$ is compact (see \cite[Theorem~7.14 and Proposition~7.18]{barlow2006diffusions}). 

Next, let $\{\theta_{i} \}_{i \in S} \subset (0,1)$ and assume 
$\sum_{i \in S}\theta_{i}=1$. As above, we define 
$\theta_{w} \coloneqq \theta_{w_{1}} \theta_{w_{2}} \cdots \theta_{w_{n}}$. 
Let $\mu$ be the unique probability measure on $X$ such that 
$\mu(F_{w}(X))=\theta_{w}$ for any $w \in W_{n}, n \in \N$. Then, 
$\mathop{\mathrm{supp}}\mu =X$ holds and $(\mathcal{E},\mathcal{F})$ is 
a strongly local Dirichlet form on $L^2(X,\mu)$ (see \cite[Theorem~7.14]{barlow2006diffusions}), where $\mathop{\mathrm{supp}}\mu$ denotes the support of the measure $\mu$. Also, for any $n \in \N$, 
we define a probability measure $\mu_{n}$ on $V_{n}$ by
\begin{align}
\mu_{n}(\{ x\}) \coloneqq
(\#V_{0})^{-1}
\sum_{w \in W_{n}}
\theta_{w}
1_{F_{w}(V_{0})}(x)
 > 0,\, x \in V_{n}.
\end{align}
Then, $\{\mu_{n} \}_{n \in \N}$ weakly converges to $\mu$ (see \cite[Lemma~5.29]{barlow2006diffusions}).

Furthermore, Condition \ref{assmp3} holds in this case 
(see \cite{hino2014geodesic}). Therefore, 
we can apply Theorem~\ref{mainfirst}
when we choose $\{h_{i} \}_{i=1}^{N} \subset \mathcal{F}$ and choose sufficiently small $\{b_{i} \}_{i=1}^{N} \subset C(X)$ which satisfy Condition \ref{assmp1}. 
(Of course, we can also choose them for Condition \ref{assmp2}). In other words, we have verified that Theorem~\ref{mainfirst} can be applied to the standard approximations of p.c.f.\ self-similar fractals and the approximating processes are finite-state continuous-time Markov chains on the vertex sets $V_{n}$.

Finally, in the case of the Sierpi\'nski gasket SG, we compare our result with previous studies. 
Let $p_1=(\frac{1}{2},\frac{\sqrt{3}}{2}), p_2=(0,0), p_3=(1,0)$ and 
$F_{i}(x) =\frac{x+p_{i}}{2}$ for $i=1, 2, 3$. Then, $(\mathrm{SG}, \{F_{i}\}_{i=1}^{3})$ is a p.c.f.\ self-similar structure and $V_{0}=\{p_{1}, p_{2}, p_{3}\}$. Let $r_1=r_2=r_3=\frac{3}{5}$ and $c_{0,x,y}=1$ for distinct
$x,y \in V_{0}$. Then, it is well-known that $\mathrm{Tr}(\mathcal{E}^{1}\lvert V_{0})=\mathcal{E}^{0}$ holds. Also, for any $n \in \N$ and $x,y \in V_{n} \,(x \neq y)$, we have
\begin{equation}
c_{n,x,y}=
\begin{cases}
(\frac{5}{3})^{n} & \text{if $\{x, y\} \in E_{n}$,}\\
0 &  \text{if $\{x, y\} \notin E_{n}$,}
\end{cases}
\end{equation}
where 
\begin{align}
E_{n}=\{\{x, y\} \subset V_{n} \mid x \neq y \text{ and there exists } w \in W_{n} 
\text{ such that } x,y \in  F_{w}(V_{0})\}.
\end{align}
Also, let $x \overset{n}{\sim} y$ denote the relation defined by $\{x, y\} \in E_{n}$.
Let $\theta_1=\theta_2=\theta_3=\frac{1}{3}$. Then, we have 
\begin{equation}
\mu_{n}(\{ x\})=
\begin{cases}
(\frac{1}{3})^{n+1} & \text{if $x \in  V_{0}$,}\\
2(\frac{1}{3})^{n+1} &  \text{if $x \in V_{n} \setminus V_{0}$.}
\end{cases}
\end{equation}

We first consider the case of no non-symmetric terms, which is well-known by previous research. Let $\tilde{Y}_{n}$ be the Markov process on $V_{n}$ associated with $(\mathcal{E}^{n},l(V_{n}))$. Then, by Remark~\ref{markovchainrepresentation}, the holding time of $\tilde{Y}_{n}$ at $x \in V_{n}$ is exponentially distributed with parameter $\tilde{q}_{n}(x)$ and $\tilde{Y}_{n}$ jumps from $x$ to $y$ with probability $\tilde{\pi}_{n}(x,y)$, where
\begin{align}
\tilde{q}_{n}(x)
&=\mu_{n}(x)^{-1}
\sum_{y \in V_{n}\setminus \{x\}}c_{n,x,y} \\
&=6\cdot5^{n},
\end{align}
and 
\begin{align}
\tilde{\pi}_{n}(x,y)
&=\frac{c_{n,x,y}}{\sum_{z \in V_{n}\setminus \{x\}} 
c_{n,x,z}} \\
&=
\begin{cases}
\frac{1}{2} & \text{if $\{x, y\} \in E_{n}$ and $x \in V_{0}$,}\\
\frac{1}{4} & \text{if $\{x, y\} \in E_{n}$ and $x \in V_{n} \setminus V_{0}$,} \\
0& \text{if $\{x, y\} \notin E_{n}$.}
\end{cases}
\end{align}
We obtain the Brownian motion on SG as the limit of $\{\tilde{Y}_{n}\}_{n \in \N}$.

Next, we observe the case with non-symmetric terms. Let $Y_{n}$ be the Markov process on $V_{n}$ associated with $(\mathcal{A}^{n},l(V_{n}))$. Then, by Remark~\ref{markovchainrepresentation}, the holding time of $Y_{n}$ at $x \in V_{n}$ is exponentially distributed with parameter $q_{n}(x)$ and $Y_{n}$ jumps from $x$ to $y$ with probability $\pi_{n}(x,y)$, where 
\begin{align}
q_{n}(x)
&=\mu_{n}(x)^{-1}
\sum_{y \in V_{n}\setminus \{x\}}
c_{n,x,y}(1+\eta(x,y)) \\
&=
\begin{cases}
3 \cdot 5^{n}\sum_{y \overset{n}{\sim} x}(1+\eta(x,y)) & \text{if $x \in V_{0}$,}\\
\frac{3}{2} \cdot 5^{n}\sum_{y \overset{n}{\sim} x}(1+\eta(x,y)) & \text{if $x \in V_{n} \setminus V_{0}$,} 
\end{cases}
\end{align}
\begin{align}
\pi_{n}(x,y)
&=\frac{c_{n,x,y}(1+\eta(x,y))}{\sum_{z \in V_{n}\setminus \{x\}}
c_{n,x,z}(1+\eta(x,z))} \\
&=
\begin{cases}
\frac{1+\eta(x,y)}{\sum_{z \overset{n}{\sim} x}
(1+\eta(x,z))} & \text{if $\{x, y\} \in E_{n}$,}\\
0& \text{if $\{x, y\} \notin E_{n}$,}
\end{cases}
\end{align}
and
\begin{align}
    \eta(x,y) = \frac{1}{2}\sum_{i=1}^{N}b_{i}(x)(h_{i}(x)-h_{i}(y)).
\end{align}
Thus, $q_{n}(x)$ and $\pi_{n}(x,y)$ are calculated by perturbing $\tilde{q}_{n}(x)$ and $\tilde{\pi}_{n}(x,y)$ with $\eta$, which might be thought of as the sum of gradients along $h_{i}$. Thus, non-symmetric perturbation appears as a directional modification of their jump rates and holding times. Although such perturbations become arbitrarily small as $n$ tends to $\infty$ due to the uniform continuity of $h_{i}$ since $\eta(x,y)$ affects these only when $x \overset{n}{\sim} y$ holds, their cumulative contribution survives in the limiting semi-Dirichlet forms through the drift terms and we can still obtain a Markov process as the limit of $\{{Y}_{n}\}_{n \in \N}$ with such perturbations. This gives a concrete interpretation 
of the semi-Dirichlet forms constructed in this paper. Thus, $\eta$ plays the role of a drift term and the resulting limit process becomes a diffusion process.
\section*{Acknowledgement}
The author is grateful to Professor Masanori Hino for his invaluable guidance and support.

\end{document}